\documentclass[leqno,11pt]{amsart}
\usepackage{amsfonts,amsmath,amssymb,color,enumerate,graphicx,psfrag,xcolor}

\newtheorem{theorem}[equation]{Theorem}

\newtheorem{lemma}[equation]{Lemma}
\newtheorem{proposition}[equation]{Proposition}

\numberwithin{equation}{section}

\begin{document}

\newcommand\1{\text{\rm 1}\hspace{-.95mm}\text{\rm I}}
\newcommand\C{\mathbb{C}}
\renewcommand\epsilon{\varepsilon}
\newcommand\N{\mathbb{N}}
\newcommand\R{\mathbb{R}}
\newcommand\ssf{\hspace{.25mm}}
\newcommand\ssb{\hspace{-.25mm}}
\newcommand\supp{\operatorname{supp}}
\newcommand\Z{\mathbb{Z}}
\newcommand\const{\operatorname{const.}}

\title[Riesz transforms and  Hardy spaces $H^1$]{Riesz transforms characterizations \\
of  Hardy spaces $H^1$ \\
for the rational  Dunkl setting \\
and multidimensional Bessel operators}

\author[J. Dziuba\'nski]{Jacek Dziuba\'nski}
\address{Uniwersytet Wroc\l awski,
Instytut Matematyczny,
Pl. Grunwaldzki 2/4,
50-384 Wroc\l aw,
Poland}
\email{jdziuban@math.uni.wroc.pl}



\thanks{
Research  supported by the Polish National Science Center
(Narodowe Centrum Nauki, grant DEC-2012/05/B/ST1/00672)
and by the University of Orl\'eans.}

%\date{\today}

\begin{abstract}
We characterize the Hardy space  $H^1$ in the rational Dunkl setting associated with the reflection group $\mathbb Z_2^n$
by means of  Riesz transforms. As a corollary we obtain a Riesz transform characterization of $H^1$ for product of Bessel operators in $(0,\infty)^n$.

\end{abstract}

\maketitle

\section{Introduction and statement of the result}

The theory of Dunkl  operators had its origin in a series of  seminal  works  \cite{Dunkl0}--\cite{Dunkl3} and was developed  by many mathematicians afterwards. The Dunkl operators form a commuting system of differential-difference  operators associated with a finite group of reflections. We refer the reader to
the lecture notes \cite{Roesler3, RoeslerVoit} and references therein for the rational Dunkl theory and to \cite{Opdam} for the trigonometric Dunkl theory.

In the present paper, on the Euclidean space $\R^n$, $n\geq 1$, we consider the Dunkl operators
\begin{equation*}
D_jf(\mathbf{x})
=\tfrac{\partial}{\partial x_j} f(\mathbf{x})
+\tfrac{k_j}{x_j}\bigl[f(\mathbf{x})\ssb-\ssb f(\sigma_j\ssf\mathbf{x})\bigr]
\qquad(j\!=\!1,2,\dots,n)
\end{equation*}
associated with the reflections
\begin{equation}\label{Reflections}
\sigma_j\ssf(x_1,x_2,\dots,x_j,\dots,x_n)=(x_1,x_2,\dots,-x_j,\dots,x_n)
\end{equation}
and the multiplicities \ssf$k_j\!\ge\ssb0$\ssf.
Their joint eigenfunctions form the Dunkl kernel
\begin{equation}\label{DunklKernelProduct}
\mathbf{E}(\mathbf{x},\mathbf{y})
=\prod\nolimits_{\ssf j=1}^{\,n}\ssb E_{k_j}(x_j,y_j)\,,
\end{equation}
\begin{equation}\label{eq1.3}
 D_j\mathbf E(\cdot,\mathbf y)(\mathbf x)=y_j\mathbf E (\mathbf x,\mathbf y),
\end{equation}
where
\begin{equation}\begin{aligned}\label{DunklKernel1D}
E_k(x,y)&=\tfrac{\Gamma(k\ssf+\frac12)}{\Gamma(k)\,\Gamma(\frac12)}
\int_{-1}^{+1}\hspace{-1mm}(1\!-\ssb u)^{k-1}\ssf(1\!+\ssb u)^k\,e^{\,x\ssf y\ssf u}\,du\,\\
&=2^{k-1\slash 2} \Gamma\Big(k+\frac{1}{2}\Big)|xy|^{\frac{1}{2}-k}\Big(I_{k-1\slash 2}(|xy|)+\text{\rm sgn}(xy)I_{k+1\slash 2}(|xy|)\Big)
\end{aligned}\end{equation}
(see for instance \cite[p. 107, Example\;2.1]{Roesler3}).
Here $I_\nu(x)$  is the modified Bessel function \ssf (see, e.g., \cite{NIST, Watson},).
Notice that \,$\mathbf{E}(\mathbf{x},\mathbf{y})\ssb
=\ssb e^{\ssf\langle\mathbf{x},\mathbf{y}\rangle}$
\ssf if all multiplicities \ssf$k_j$ vanish.

The Dunkl Laplacian
\begin{equation*}
\mathbf{L}f(\mathbf{x})
=\sum\nolimits_{\ssf j=1}^{\,n}\ssb D_j^{\ssf2}f(\mathbf{x})
=\sum\nolimits_{\ssf j=1}^{\,n}\ssf\Bigl\{
\bigl(\tfrac\partial{\partial\ssf x_j}\bigr)^2\ssb f(\mathbf{x})\ssb
+\ssb\tfrac{2\ssf k_j}{x_j}\ssf\tfrac{\partial}{\partial\ssf x_j}f(\mathbf{x})\ssb
-\ssb\tfrac{k_j}{x_j^2}\ssf\bigl[\ssf f(\mathbf{x})\!
-\!f(\sigma_j\ssf\mathbf{x})\ssf\bigr]\Bigr\}
\end{equation*}
is the infinitesimal generator of the heat semigroup
$\{e^{\,t\,\mathbf{L}}\}_{t>0}$,
which acts by linear self-adjoint operators on $L^2(\R^n,d\boldsymbol{\mu})$
and by linear contractions on $L^p(\R^n,d\boldsymbol{\mu})\ssf$,
for every $1\ssb\le\ssb p\ssb\le\ssb\infty$\ssf,
where
\begin{equation}\label{ProductMeasure}
d\boldsymbol{\mu}(\mathbf{x})
=d\mu_1(x_1)\ssf\dots\,d\mu_n(x_n)
=|x_1|^{\ssf2\ssf k_1}\dots\,|x_n|^{\ssf2\ssf k_n}\,
dx_1\ssf\dots\ssf dx_n.
\end{equation}
Clearly,
\begin{equation*}
 \mathbf L \mathbf E(\cdot, \mathbf y)(\mathbf x)=|\mathbf y|^2 \mathbf E(\mathbf x,\mathbf y).
\end{equation*}
The heat semigroup, which is strongly continuous on $L^p(\mathbb R^n,d\boldsymbol \mu)$ for $1\leq p<\infty$,  consists of integral operators
\begin{equation*}
e^{\,t\ssf\mathbf{L}}f(\mathbf{x})
=\int_{\ssf\R^n}\,
\mathbf{h}_{\ssf t}(\mathbf{x},\mathbf{y})\,f(\mathbf{y})\,d\boldsymbol{\mu}(\mathbf{y})\,
\end{equation*}
associated with the heat kernel
\begin{equation}\label{HeatKernelProduct}
\mathbf{h}_{\ssf t}(\mathbf{x},\mathbf{y})=
\mathbf{c}_{\ssf\mathbf{k}}^{-1}\,t^{-\frac{\mathbf{N}}2}\,
e^{-\frac{|\mathbf{x}|^2+\ssf|\mathbf{y}|^2}{4\,t}}\,
\mathbf{E}\bigl(\tfrac{\mathbf{x}}{\sqrt{2\ssf t\ssf}},
\tfrac{\mathbf{y}}{\sqrt{2\ssf t\ssf}}\bigr)\ssf,
\end{equation}
see, e.g.,  \cite{Roesler2}, where
\begin{equation}\label{HomogeneousDimension}
\mathbf{N}=n+\sum\nolimits_{\ssf j=1}^{\,n}\ssb2\,k_j
\end{equation}
is the homogeneous dimension and
\begin{equation*}
\mathbf{c}_{\ssf\mathbf{k}}\ssf=\,2^{\frac{\mathbf{N}}2}\ssb
\int_{\ssf\R^n}\,e^{-\frac{|\mathbf{x}|^2}2}\ssf\,d\boldsymbol{\mu}(\mathbf{x})\,
=\,2^{\ssf\mathbf{N}}\,\prod\nolimits_{\ssf j=1}^{\,n}\ssb\Gamma(k_j\!+\!\tfrac12).
\end{equation*}

The Dunkl  transform is defined  by
\begin{equation}\label{FourierTransform}
\mathcal{F}\ssb f(\boldsymbol{\xi})
=\ssf\mathbf{c}_{\ssf\mathbf{k}}^{-1}\!
\int_{\ssf\R^n}f(\mathbf{x})\,
\mathbf{E}(\mathbf{x},\ssb-\ssf i\ssf\boldsymbol{\xi})\,d\boldsymbol{\mu}(\mathbf{x})\,.
\end{equation}
It is an isometric isomorphism of \ssf$L^2(\R^n,d\boldsymbol{\mu})$ onto itself
with  the inversion formula:
\begin{equation*}
f(\mathbf{x})=\mathcal{F}^{\ssf2\ssb}f(-\mathbf{x})\,
\end{equation*}
(see, e.g., \cite{Dunkl3}, \cite{dJ}).

The Hardy space \ssf$H^1_{{\rm max},\, \mathbf L}$ associated with $\mathbf L$ is the set of all functions
\ssf$f\!\in\!L^1(\R^n,d\boldsymbol{\mu})$
\ssf whose maximal heat function
\begin{equation}\label{MaximalHeatOperator}
\mathbf{h}_{\ssf*}f(\mathbf{x})
=\ssf\sup\nolimits_{\,t>0}\,\Bigl|\ssf\int_{\ssf\R^n}
\mathbf{h}_{\ssf t}(\mathbf{x},\mathbf{y})\,f(\mathbf{y})\ssf d\boldsymbol{\mu}(\mathbf{y})\Bigr|
\end{equation}
belongs to \ssf$L^1(\R^n,d\boldsymbol{\mu})$
and the norm is given by
\begin{equation*}
\|f\|_{H^1_{\rm{max},\mathbf L}}=\ssf\|\ssf\mathbf{h}_{\ssf*}f\ssf\|_{L^1(\mathbb R^n,\, d\boldsymbol\mu)}\ssf.
\end{equation*}

Now we turn  to the atomic definition of the Hardy space $H^1$.
Notice that $\R^n$,
equipped with the Euclidean distance
\ssf$d\ssf(\mathbf{x},\mathbf{y})\ssb=\ssb|\ssf\mathbf{x}\!-\!\mathbf{y}\ssf|$
\ssf and with the measure \ssf$\boldsymbol{\mu}$\ssf,
is a space of homogeneous type in the sense of Coifman-Weiss  \cite{CoifmanWeiss}.
An atom is a measurable function \,$a\ssb:\ssb\R^n\!\to\ssb\C$
\,such that
\begin{itemize}
\item[$\bullet$]
$\;a$ \ssf is supported in a ball \ssf$B$\ssf,
$\vphantom{\displaystyle\int}$

\item[$\bullet$]
$\;\|a\|_{L^\infty}\!\lesssim\boldsymbol{\mu}\ssf(B)^{-1}$\ssf,

\item[$\bullet$]
$\;\displaystyle\int_{\ssf\R^n}\hspace{-1mm}a(\mathbf{x})\, d\boldsymbol{\mu}(\mathbf{x}) =0$\ssf.
\end{itemize}
By definition, the atomic Hardy space \ssf$H^1_{\text{atom}}$
\ssf consists of all functions \ssf$f\!\in\!L^1(\R^n,d\boldsymbol{\mu})$
which can be written as
\ssf$f\hspace{-.4mm}=\ssb\sum_{\ssf\ell}\ssb\lambda_{\ssf\ell}\,a_{\ssf\ell}$\ssf,
where the \ssf$a_\ell$'s are atoms and
\ssf$\sum_{\ssf\ell}\ssb|\lambda_{\ssf\ell}|\!<\!+\infty$\ssf,
and the norm is given by
\begin{equation*}
\|f\|_{H^1_\text{atom}}=\,\inf\,\sum\nolimits_{\ssf\ell}|\lambda_\ell|\,,
\end{equation*}
where the infimum is taken over all atomic decompositions of \ssf$f$.

Hardy spaces on spaces of homogeneous type (see, e.g., \cite{CoifmanWeiss}, \cite{MaciasSegovia}, \cite{Uchiyama}) are extensions of the classical real Hardy spaces on $\mathbb R^n$. For characterizations and properties of the classical Hardy spaces we refer the reader to the original works  \cite{BurkholderGundySilverstein}, \cite{FeffermanStein}, \cite{SteinWeiss}, \cite{Coifman}. More information are given in the book \cite{Stein} and references therein.

Hardy spaces associated with the Dunkl operator $\mathbf L$ were studied  in \cite{ABDH}. The following theorem  was proved there.

 \begin{theorem}
The spaces $H^1_{\rm{max}, \mathbf L}$ and $H^1_\text{\rm atom}$ coincide and the norms
$\|f\|_{H^1_{\rm{max},\mathbf L}}$ and $\|f\|_{H^1_\text{\rm atom}}$ are equivalent.
\end{theorem}

The present paper is a continuation of \cite{ABDH} and deals with  the Riesz transforms characterization of $H^1_{{\rm max},\mathbf L}$.  We define the Riesz transforms in the Dunkl setting putting
$$ \mathcal R_j=D_j(-\mathbf L)^{-1\slash 2}.$$
The operators $\mathcal R_j$ can be expressed as the  Dunkl  multiplier operators, namely,
\begin{equation}\begin{aligned}
 \mathcal R_jf(\mathbf x) &=D_j(- \mathbf L)^{-1\slash 2}f(\mathbf x)
 =D_j\int_{\mathbb R^n} \frac{1}{|\boldsymbol \xi|} \mathbf E(\mathbf x,i\boldsymbol \xi)
 \mathcal Ff(\boldsymbol\xi)\, d\boldsymbol\mu (\boldsymbol  \xi)\\
 &=\int_{\mathbb R^n} i\, \frac{\xi_j}{|\boldsymbol \xi|} \mathbf E(\mathbf x,i\boldsymbol \xi)
 \mathcal F f(\boldsymbol \xi) d\boldsymbol \mu(\boldsymbol \xi).\\
\end{aligned}\end{equation}

Our  main result is the following theorem which is an analogue of the result about the characterization of the classical Hardy spaces by the classical Riesz transforms $\frac{\partial}{\partial x_j}(-\Delta)^{-1\slash 2}$ (see, e.g., \cite[Chapter III, Section 4]{Stein}).

\begin{theorem}\label{Theorem1} Let $f \!\in\!L^1(\R^n,d\boldsymbol{\mu})$. Then $f\in H^1_{\rm{max},\mathbf L}$
if and only if $\mathcal R_jf\in \!L^1(\R^n,d\boldsymbol{\mu})$ for $j=1,2,...,n$. Moreover,
there exists a constant \,$C\!>\!0$ such that
\begin{equation}\label{eq222}
C^{-1}\,\|f\|_{H^1_{\rm{max},\mathbf L}}\ssb\le \| f\|_{L^1(\mathbb R^n,\, d\boldsymbol\mu)}+\sum_{j=1}^n \| \mathcal R_jf\|_{L^1(\mathbb R^n,\, d\boldsymbol \mu) } \leq  C\,\| f\|_{H^1_{\rm{max},\mathbf L}}\ssf.
\end{equation}
\end{theorem}

Let us  emphasize that Theorem \ref{Theorem1} implies a Riesz transform characterization of the Hardy space $H^1_{{\rm max}, \mathbb L}$ associated with multidimensional Bessel operator $\mathbb L$. To be more precise, on $(0,\infty)^n$ equipped with the measure $d\boldsymbol\mu$ we consider the Bessel operator
\begin{equation*}\label{Bessel}
{\mathbb L}=
\sum_{j=1}^n \Big(\frac{\partial^2}{\partial x_j^2}+\frac{2k_j}{x_j}\frac{\partial}{\partial x_j}\Big)
\end{equation*}
and the  associated semigroup $\{e^{t\mathbb L}\}_{t>0}$. The action of the semigroup $e^{t\mathbb L}$ on functions is given by integration against the heat kernel $\mathbb H_t(\mathbf x,\mathbf y)$, namely,
\begin{equation}\label{productBessel}
 e^{t\mathbb L} f(\mathbf x)=\int_{(0,\infty)^n } \mathbb H_t(\mathbf x,\mathbf y) f(\mathbf y) d\boldsymbol\mu (\mathbf y),
\end{equation}
where $\mathbb H_t(x,y)=\prod_{j=1}^n {\tt h}_t^{[j]}(x_j,y_j) $,
\begin{equation}\label{heatBessel}{\tt h}_t^{[j]}(x_j,y_j)= (2 t )^{-1}\exp
 (-(x_j^2+y_j^2)\slash 4t) I_{k_j-1\slash 2}\Big(\frac{x_jy_j}{2t}\Big)(x_jy_j)^{-k_j +1\slash 2},
 \end{equation}
 (see \cite{Watson}). We define the Hardy space (see, e.g., \cite{DziubanskiPreisnerWrobel})
 $$H^1_{\rm{max},\mathbb L}=\Big\{ f\in L^1((0,\infty)^n, d\boldsymbol\mu):\big\| \sup_{ t>0} | e^{t\mathbb L}f|\big\|_{L^1((0,\infty)^n, d\boldsymbol\mu)}=\| f\|_{H^1_{{\rm max},\mathbb L}}<\infty\Big\}.$$
Let $R_j=\partial_{x_j}\mathbb L^{-1\slash 2}$ denote the Riesz transform associated with $\mathbb L$.
Now we state our second main result.
\begin{theorem}\label{TheoremBessel} Assume that  $f\in L^1((0,\infty)^n, d\boldsymbol\mu)$. Then $f$  belongs to $H^1_{\rm{max},\mathbb L}$ if and only if $R_jf\in L^1((0,\infty)^n, d\boldsymbol\mu)$ for $j=1,2,...,n$. Moreover, there is a constant $C>0$ such that
\begin{equation} C^{-1} \| f\|_{H^1_{\rm{max},\mathbb L}} \leq \| f\|_{L^1((0,\infty)^n, d\boldsymbol\mu)} +\sum_{j=1}^n \| R_jf\|_{L^1((0,\infty)^n, d\boldsymbol\mu)}\leq C \| f\|_{H^1_{\rm{max},\mathbb L}}.
 \end{equation}
\end{theorem}

\medskip

\section{Poisson semigroup}

The  Poisson semigroup  $\{ P_t\}_{t>0}$ in the Dunkl setting  is defined by:
$$ P_tf(\mathbf x)=e^{-t\sqrt{-\mathbf L}}f(\mathbf x)= \mathcal F^{-1}(e^{-t|\boldsymbol \xi|}\mathcal Ff(\boldsymbol \xi))(\mathbf x)=\int_{\mathbb R^n} P_t(\mathbf x,\mathbf y)\, d\boldsymbol\mu (\mathbf y),$$
where the  associated Poisson kernel is given by
$$ P_t(\mathbf x, \mathbf y)= c\int_{\mathbb R^n} \mathbf E ( \mathbf x,i\boldsymbol \xi)e^{-t|\boldsymbol \xi|} \mathbf E (\mathbf y,-i\boldsymbol \xi)d\boldsymbol \mu (\boldsymbol \xi).
$$
 By the subordination formula
\begin{equation}\label{subordination}
P_t(\mathbf x,\mathbf y)=\frac{1}{\sqrt{\pi}}\int_0^\infty e^{-u}\mathbf h_{t^2\slash 4u}(\mathbf x,\mathbf y)\frac{du}{\sqrt{u}}.
\end{equation}
It easily follows from (\ref{subordination}), (\ref{HeatKernelProduct}), (\ref{DunklKernelProduct}) and (\ref{DunklKernel1D}) that $P_t(\mathbf x,\mathbf y)=P_t(\mathbf y,\mathbf x)$ is a positive smooth function of the $(t,\mathbf x,\mathbf y)$ variables.  We shall also show (see Appendix) that for every $1\leq p<\infty$ and $t>0$ there is a constant $C_{p,t}$ such that
\begin{equation}\label{Lp}
\sup_{\mathbf x\in\mathbb R^n} \int_{\mathbb R^n} P_t(\mathbf x,\mathbf y)^pd\boldsymbol\mu(\mathbf y)\leq C_{p,t}.
\end{equation}
Let $\mathcal L=\frac{\partial^2}{\partial t^2}+\mathbf L$. Then
\begin{equation}\label{harmonic} \mathcal L P_tf (\mathbf x)=0.
\end{equation}
Let
$$P_*f(\mathbf x)=\sup_{t>0} |P_tf(\mathbf x)|$$
be the maximal operator associated with $\{P_t\}_{t>0}$.

In order to prove Theorem \ref{Theorem1}  we shall use Theorem \ref{propositionP} and Proposition \ref{propositionP2}. The proofs of them together with basic properties of $P_t(\mathbf x,\mathbf y)$ are presented in the appendix.

\begin{theorem}\label{propositionP} (a)
 Let $f\in L^1(\mathbb R^n,\, d\boldsymbol \mu)$. Then $f\in H^1_{{\rm max},\mathbf L}$ if and only if the maximal function $P_*f$
 belongs to $L^1 (\mathbb R^n,\, d\boldsymbol \mu)$. Moreover,
 $$ \| P_*f\|_{L^1(\mathbb R^n,\, d\boldsymbol \mu)} \sim \| f\|_{H^1_{\rm{max},\mathbf L}}. $$

(b) For every $1<p\leq \infty$ the maximal operator $P_*$ is bounded on $L^p(\mathbb R^n,\, d\boldsymbol \mu)$.
\end{theorem}

\begin{proposition}\label{propositionP2} (a) Assume that  $g\in L^1_{\rm loc}(\mathbb R^n, d\boldsymbol \mu)$ and $\lim_{|\mathbf x|\to\infty } |g(\mathbf x)|=0$.
 Then
$$\lim_{(|\mathbf x|+t)\to\infty } P_tg(\mathbf x)=0.$$
(b) If $f\in L^1(\mathbb R^n,d\boldsymbol \mu)$,  then for every $\varepsilon >0$ we have
\begin{equation}\label{DD1} \lim_{(|\mathbf x|+t)\to\infty} P_{t+\varepsilon} f(\mathbf x)=0.
\end{equation}

\end{proposition}

\section{Key lemma}
The following lemma, which is perhaps interesting in its own, will play a crucial role in the proof of Theorem \ref{Theorem1}.
\begin{lemma}\label{lemma22}
 For every positive integer $n$ and every  $\varepsilon>0$ there is $\delta>0$ such that for any matrix
 $$B=\left[\begin{array}{cccc}
b_{0,0} & b_{0,1}&...& b_{0,n}\\
b_{1,0} & b_{1,1}&...&b_{1,n} \\
\  & \ &...&\  \\
b_{n,0} & b_{n,1}&...&b_{n,n} \\
\end{array}\right]
$$  with real entries we have
\begin{equation}\label{eq2.30}
 \| B\|^2\leq (1-\delta) \| B\|_{\text{\rm HS}}^2+ \varepsilon \Big( (\text{\rm tr}\,  B)^2+\sum_{i<j} (b_{i,j}-b_{j,i})^2\Big).
\end{equation}
\end{lemma}
Here $\| B\|=\sup_{\boldsymbol x\in\mathbb R^{n+1}, \ \boldsymbol \|\boldsymbol x\|=1} \| B\boldsymbol x\|$ is the ordinary norm of $B$ and $\| B\|_{\rm{HS}}=(\sum_{j=0}^n\sum_{\ell =0}^n b_{j,\ell}^2)^{1\slash 2}$ is the Hilbert-Schmidt norm.
\begin{proof}
Let $S=(s_{i,j})_{i,j=0,1,...,n}$ and $A=(a_{i,j})_{i,j=0,1,...,n}$ denote any symmetric and antisymmetric matrix respectively. It is clear that
\begin{equation}\label{HS1}
\| A+S\|_{\text{\rm HS}}^2=\| A\|_{\text{\rm HS}}^2+ \| S\|_{\text{\rm HS}}^2,
\end{equation}
\begin{equation}\label{eqHS2}
\sum_{i<j}(a_{i,j}-a_{j,i})^2=2\| A\|_{\text{\rm HS}}^2.
\end{equation}
 It is known (see e.g., Section 3.1.2 of Chapter VII of \cite{SteinSing}) that (\ref{eq2.30}) holds for symmetric trace zero  matrixes. Observe that it also holds for antisymmetric matrixes with $\delta=2\varepsilon$. Indeed, from  (\ref{eqHS2}), we get
  \begin{equation*}\begin{split}
   \| A\|^2\leq \| A\|_{\text{\rm HS}}^2 = ( 1-2\varepsilon)\| A\|_{\text{\rm HS}}^2 + \varepsilon \sum_{i<j}(a_{i,j}-a_{j,i})^2.
  \end{split}\end{equation*}
We claim  that for any fixed
 $\varepsilon >0$ and $A$, $S$ such that   $\| A\|_{\rm HS}^2\geq \frac{1}{\varepsilon}$,  $\| S\|_{\rm HS}^2 = 1$,  we have
\begin{equation}\label{eq2.301} \| A+S\|^2 \leq (1-\varepsilon )\| A+S\|_{{\rm HS}}^2+ 2\varepsilon \| A\|_{\text{HS}}^2.
\end{equation}
To see (\ref{eq2.301}) we utilize (\ref{HS1}) and obtain
\begin{equation*}
\begin{split}
 \| A+S\|^2 &\leq \| A+S\|_{\text{\rm HS}}^2 =\| A\|_{\text{\rm HS}}^2+\| S\|_{\text{\rm HS}}^2\\
 &= (1-\varepsilon)\|A\|_{\text{\rm HS}}^2 + \varepsilon \| A\|_{\text{\rm HS}}^2 + 1\\
 &\leq (1-\varepsilon)\|A\|_{\text{\rm HS}}^2 +\varepsilon
 \| A\|_{\text{\rm HS}}^2  + \varepsilon \| A\|_{\text{\rm HS}}^2\\
 &\leq (1-\varepsilon)\|A+S\|_{\text{\rm HS}}^2 +2\varepsilon \| A\|_{\text{\rm HS}}^2.
\end{split}
\end{equation*}
 Since  (\ref{eq2.30}) is homogeneous of degree 2, that is, $\| tB\|^2=t^2\| B\|^2$, and
 $$(1-\delta) \| tB\|_{\text{\rm HS}}^2+ \varepsilon \Big( (\text{\rm tr} \, t B)^2+\sum_{i<j} (tb_{i,j}-tb_{j,i})^2\Big)=t^2\Big( (1-\delta) \| B\|_{\text{\rm HS}}^2+ \varepsilon \Big( (\text{\rm tr} B)^2+\sum_{i<j} (b_{i,j}-b_{j,i})^2\Big)\Big),$$
 it suffices to prove (\ref{eq2.30}) for $B=A+S$ with $S$ running over the unit sphere in the Hilbert-Schmidt norm, that is, $\| S\|_{\text{\rm HS}}^2=1$.
 Assume that (\ref{eq2.30}) does not hold. Then there is  $ \varepsilon>0$ such that for every  $\delta_n =\frac{1}{n}$ there are  $A_n$ and $S_n$, $\| S_n\|_{\rm HS}^2 = 1$, such that
\begin{equation}\label{eq4}
 \| A_n+S_n\|^2 > (1-\delta_n) \| A_n+S_n\|_{\text{\rm HS}}^2 + 2 \varepsilon \| A_n\|_{\text{\rm HS}}^2+\varepsilon (\text{\rm tr}\, S_n)^2.
\end{equation}
  It follows from (\ref{eq2.301}) that  $\| A_n\|_{\rm HS}^2\leq \frac{1}{\varepsilon }$ for large  $n$. Thus  $S_n$ and $A_n$ are in compact sets. There is a subsequence $n_k$ such that  $S_{n_k}$ and  $A_{n_k}$ converge to  $S$ and $A$ respectively.   Moreover,  $\| S\|_{\text{\rm HS}}^2=1$.
 Passing to limit in  (\ref{eq4}) as $k\to\infty$,  we obtain
\begin{equation}\label{eq5}
 \| A+S\|^2 \geq  \| A+S\|_{\text{\rm HS}}^2 + 2\varepsilon \| A\|_{\text{\rm HS}}^2+\varepsilon (\text{\rm tr}\, S)^2.
\end{equation}
The inequality (\ref{eq5}) implies that   $A=0$ and  $\text{tr}\, S=0$. Hence
$\| S\|^2\geq \| S\|_{\text{\rm HS}}^2$, which is impossible for a nonzero symmetric matrix $S$ with $\text{\rm tr}\, S=0$.
\end{proof}

\section{Riesz transforms and Cauchy-Riemann equations}

For a function $f\in L^1(\mathbb R^n, d\boldsymbol \mu)$ such that
$\mathcal R_jf\in L^1(\mathbb R^n, d\boldsymbol \mu)$ we define the functions
\begin{equation}\begin{split}\label{vectoru}
 u_0(t,\mathbf x)&=P_tf(\mathbf x)=
\int_{\mathbb R^n}
                 e^{-t|\boldsymbol \xi|} \mathcal F f(\boldsymbol \xi) \mathbf E(\mathbf x,i\boldsymbol \xi)
                 d\boldsymbol\mu (\boldsymbol \xi), \\
 u_j(t,\mathbf x)&=-P_t(\mathcal R_jf)(\mathbf x)
=-\int_{\mathbb R^n} i\,\frac{\xi_j}{|\boldsymbol \xi|}
                 e^{-t|\boldsymbol \xi|} \mathcal F f(\boldsymbol \xi) \mathbf E(\mathbf x,i\boldsymbol \xi)
                 d\boldsymbol\mu (\boldsymbol \xi).
               \end{split} \end{equation}
The functions $u_j$, $j=0,1,...,n$,   are $ C^\infty$ on $(0,\infty)\times \mathbb R^n$. It is easy to check using (\ref{eq1.3}) and (\ref{vectoru})
that they satisfy the following
Cauchy-Riemann type equations:
\begin{equation}\label{CR}
 \begin{aligned}
  &D_ju_0(t,\mathbf x) =\partial_t u_j(t,\mathbf x), \ j=1,...,n;\\
  &D_ju_\ell (t,\mathbf x)=D_\ell u_j(t,\mathbf x), \ \ j,\ell =1,....,n;\\
  &\partial_t u_0(t,\mathbf x)+ \sum_{j=1}^nD_ju_j(t,\mathbf x)=0.
 \end{aligned}
\end{equation}
From now we shall assume that $f$ is real-valued,  then so are $u_j$, $j=0,1,...,n$.

 Let $\mathcal G$ denote the group of reflections in $\mathbb R^n$  generated by $\sigma_j$, $j=1,...,n$.
 For $\sigma\in \mathcal G$ and a function $u(t,\mathbf x)$ defined on $(0,\infty)\times \mathbb R^n$ we denote
 $u^\sigma(t,\mathbf x)=u(t,\sigma\mathbf x)$.
 If $\mathbf u(t,\mathbf x)=(u_0(t,\mathbf x),u_1(t,\mathbf x),...,u_n(t,\mathbf x))$ satisfies (\ref{CR}),
 then so does
 $\mathbf u^\sigma(t,\mathbf x)=(u_0^\sigma(t,\mathbf x),u_1^\sigma(t,\mathbf x),...,u_n^\sigma(t,\mathbf x))$.

 Moreover, if $(u_0(t,\mathbf x),u_1(t,\mathbf x),...,u_n(t,\mathbf x))$ is of the form  (\ref{vectoru}), then
 $$ u_0^\sigma(t,\mathbf x)=P_t (f^\sigma) (\mathbf x), \ \
 u_j^\sigma (t,\mathbf x)=P_t(\mathcal R_j(f^\sigma))(\mathbf x),$$
 where $f^\sigma(\mathbf x)=f(\sigma\mathbf x)$.

For a $C^2$ function $\mathbf u(t,\mathbf x)=(u_0(t,\mathbf x),u_1(t,\mathbf x),...,u_n(t,\mathbf x))$
satisfying
(\ref{CR}) consider the function $F:(0,\infty)\times \mathbb R^n\to \mathbb R^{(n+1)\cdot |\mathcal G|}$,
$$F(t,\mathbf x)= \{\mathbf u^\sigma (t,\mathbf x)\}_{\sigma\in \mathcal G}.$$
Observe that
$|F(t,\mathbf x)|=|F(t,\sigma\mathbf x)|$ for every $\sigma\in \mathcal G$, where
$$ |F(t,\mathbf x)|^2=\sum_{\sigma\in \mathcal G} \sum_{\ell =0}^n |u^\sigma_\ell (t,\mathbf x)|^2.$$
Our main taks is to prove that
the following proposition, which is an analogue of the classical result (see, e.g., \cite[Section 3.1 of Chapter VII]{SteinSing}).
\begin{proposition}\label{subharmonic}
 There is an exponent $0<q<1$ which depends on $k_1,...,k_n$ such that  the function  $|F|^q$ is $\mathcal L$-subharmonic,
 that is, $\mathcal L(|F|^q)(t,\mathbf x)\geq 0$ on the set
 where $|F|>0$.
\end{proposition}
\begin{proof} Observe that $|F|^q$ is $C^2$ on the set where $|F|>0$.
 Let $\cdot$ denote the inner product in $\mathbb R^{(n+1)\cdot |\mathcal G|}$. In order to unify  our notation we denote the variable $t$ by $x_0$.  For $j=0,1,...,n$,  we have
 \begin{equation*}\begin{aligned}
        \partial_{x_j} |F|^q &=q|F|^{q-2}\Big((\partial_{x_j} F)\cdot F\Big)\\
        \partial_{x_j}^2 |F|^q&=q(q-2)|F|^{q-4}\Big((\partial_{x_j}F)\cdot F\Big)^2 +q|F|^{q-2}\Big((\partial_{x_j}^2F)\cdot
        F+|\partial_{x_j}F|^2\Big).
        \end{aligned}
   \end{equation*}
Recall that $|F(x_0,\mathbf x)|=|F(x_0,\sigma\mathbf x)|$. Hence,
\begin{equation}\label{eq366}
 \begin{aligned}
\mathcal L|F|^q&=  q(q-2)|F|^{q-4}\Big\{ \Big((\partial_{x_0}F)\cdot F\Big)^2+
\sum_{j=1}^n\Big((\partial_{x_j}F)\cdot F\Big)^2\Big\}\\
& \ + q|F|^{q-2}\Big\{ \Big(\partial_{x_0}^2 F+\sum_{j=1}^n\big(\partial_{x_j}^2F+
\frac{2k_j}{x_j}(\partial_{x_j} F)\big)\Big)\cdot F+|\partial_{x_0}F|^2+\sum_{j=1}^n|\partial_{x_j} F|^2 \Big\}.\\
\end{aligned}
\end{equation}
 Since $D_jD_\ell f=D_\ell D_j f$ for $f\in C^2(\mathbb R^n)$, we conclude from (\ref{CR})   that for $\ell =0,1,...,n$ and $\sigma\in \mathcal G$ we have
 $$  \partial_{x_0}^2 u^\sigma_\ell +\sum_{j=1}^n\big(\partial_{x_j}^2 u^\sigma_\ell +
\frac{2k_j}{x_j}(\partial_{x_j} u^\sigma_\ell )\big)=\sum_{j=1}^n \frac{k_j}{x_j^2}(u_\ell ^\sigma-u_\ell ^{\sigma\sigma_j}).$$
Thus,
\begin{equation}\label{eq367}\begin{split}
\Big(\partial_{x_0}^2 F+\sum_{j=1}^n\big(\partial_{x_j}^2F+
\frac{2k_j}{x_j}(\partial_{x_j} F)\big)\Big)\cdot F
&=\sum_{\sigma\in \mathcal G} \sum_{\ell =0}^n\sum_{j=1}^n
\frac{k_j}{x_j^2} (u_\ell^\sigma-u_\ell^{\sigma\sigma_j})u_\ell^\sigma\\
&= \sum_{j=1}^n \sum_{\ell=0}^n \sum_{\sigma\in \mathcal G}
\frac{k_j}{x_j^2} (u_\ell^\sigma-u_\ell^{\sigma\sigma_j})u_\ell^\sigma\\
&=\frac{1}{2}\sum_{j=1}^n \sum_{\ell =0}^n \sum_{\sigma\in \mathcal G}
\frac{k_j}{x_j^2} (u_\ell^\sigma-u_\ell^{\sigma\sigma_j})^2\\
&=\frac{1}{2}\sum_{\sigma\in \mathcal G}
\sum_{j=1}^n \sum_{\ell =0}^n
\frac{k_j}{x_j^2} (u_\ell^\sigma-u_\ell^{\sigma\sigma_j})^2.\\
\end{split}
\end{equation}
Thanks to (\ref{eq366}) and (\ref{eq367}), it suffices to prove that there is $0<q<1$ such that
\begin{equation}\label{eq3.1}
 \begin{aligned}
  &(2-q)\Big\{ \Big((\partial_{x_0}F)\cdot F\Big)^2+ \sum_{j=1}^n\Big((\partial_{x_j}F)\cdot F\Big)^2\Big\}\\
  &\leq \frac{1}{2}
  |F|^2 \sum_{\sigma\in \mathcal G}
\sum_{j=1}^n \sum_{\ell =0}^n
\frac{k_j}{x_j^2} (u_\ell^\sigma-u_\ell^{\sigma\sigma_j})^2
 +|F|^{2}\Big(|\partial_{x_0}F|^2+\sum_{j=1}^n|\partial_{x_j} F|^2\Big).
 \end{aligned}
\end{equation}
Denote
$$B_\sigma=\left[\begin{array}{cccc}
\partial_{x_0} u^\sigma_0 &\partial_{x_0}u^\sigma _1& ... &\partial_{x_0}  u_n^\sigma\\
\partial_{x_1} u_0^\sigma &\partial_{x_1} u_1^\sigma & ... &\partial_{x_1}  u_n^\sigma\\
 \  & \  & ... & \ \\
 \partial_{x_n} u_0^\sigma &\partial_{x_n} u_1^\sigma & ... &\partial_{x_n}  u_n^\sigma\\
\end{array}\right].$$
Let $\mathbf B=\{ B_\sigma\}_{\sigma\in \mathcal G}$ be matrix with $n+1$ rows and $(n+1)\cdot |\mathcal G|$ columns. It represents a linear operator from $\mathbb R^{(n+1)\cdot |\mathcal G|}$ into $\mathbb R^{n+1}$.

Observe that
$$ (2-q)\Big\{ \Big((\partial_{x_0}F)\cdot F\Big)^2+ \sum_{j=1}^n\Big((\partial_{x_j}F)\cdot F\Big)^2\Big\}
\leq (2-q)|F|^2  \| \mathbf B\|^2 ,$$
$$ | F|^{2}\Big(|\partial_{x_0}F|^2+\sum_{j=1}^n|\partial_{x_j} F|^2\Big) = |F|^2\| \mathbf B\|_{\text{\rm HS}}^2.$$
Clearly,
$$\| \mathbf B\|^2\leq \sum_{\sigma\in \mathcal G} \| B_\sigma\|^2, \ \ \ \| \mathbf B\|_{\text{\rm HS}}^2=\sum_{\sigma\in \mathcal G}\| B_\sigma\|_{\text{\rm HS}}^2.$$
Therefore the inequality  (\ref{eq3.1}) will be proven if we show that
\begin{equation}
\label{eq3.2}
(2-q)\sum_{\sigma\in \mathcal G}\|B_{\sigma}\|^2\leq \sum_{\sigma\in \mathcal G}  \| B_\sigma\|_{\text{\rm HS}}^2+ \frac{1}{2}
   \sum_{\sigma\in \mathcal G}
\sum_{j=1}^n \sum_{\ell =0}^n
\frac{k_j}{x_j^2} (u_\ell^\sigma-u_\ell^{\sigma\sigma_j})^2.
\end{equation}
Applying the Cauchy-Riemann type equations (\ref{CR}), we obtain
\begin{equation}\label{eq368}
( \text{\rm tr} B_\sigma )^2
= \Big(-\sum_{j=1}^n \frac{k_j}{x_j}(u_j^\sigma -u_j^{\sigma\sigma_j})\Big)^2
\leq \Big(\sum_{s=1}^n k_s\Big)\Big(\sum_{j=1}^n \frac{k_j}{x_j^2} (u_j^{\sigma} -u_j^{\sigma\sigma_j})^2\Big) ,
\end{equation}
\begin{equation}\label{eq369}\begin{split}
 \sum_{i<j} (\partial_{x_i}u^\sigma_{j}-\partial_{x_j}u^\sigma_i)^2 & =\sum_{j=1}^n \frac{k_j^2}{x_j^2}(u^\sigma_0-u^{\sigma\sigma_j}_0)^2 +\sum_{1\leq i<j} \Big(\frac{k_i}{x_i}
 (u_j^\sigma -u_j^{\sigma\sigma_i})-\frac{k_j}{x_j}(u^\sigma_i-u^{\sigma\sigma_j}_i)\Big)^2\\
 &\leq 2\Big(\sum_{s=1}^n k_s\Big)\Big(\sum_{i=0}^n \sum_{j=1}^n \frac{k_j}{x_j^2}(u_i^\sigma-u_i^{\sigma\sigma_j})^2\Big).
\end{split} \end{equation}
Using  Lemma \ref{lemma22} together with (\ref{eq368}) and (\ref{eq369}) we have that for every $\varepsilon >0$  there is $\delta>0$ such that
\begin{equation}\begin{split}\label{eq3.44}
 \sum_{\sigma\in \mathcal G} \| B_\sigma\|^2 &\leq (1-\delta)\sum_{\sigma\in \mathcal G}\| B_\sigma\|_{\text{\rm HS}}^2 +3\varepsilon \Big(\sum_{s=1}^n k_s\Big) \sum_{\sigma\in \mathcal G} \Big(\sum_{i=0}^n \sum_{j=1}^n \frac{k_j}{x_j^2}(u_i^\sigma-u_i^{\sigma\sigma_j})^2\Big).
\end{split}\end{equation}
Taking $\varepsilon>0$ such that $3\varepsilon  \sum_{s=1}^n k_s\leq \frac{1}{4}$ and utilizing (\ref{eq3.44}) we deduce that (\ref{eq3.2}) holds with $(1-\delta)\leq (2-q)^{-1}$.
\end{proof}

\section{Maximum principle}

On $\mathbb R_+\times \mathbb R^n$ let
$$\widetilde {\mathbb L}=\frac{\partial^2}{\partial x_0^2}+
\sum_{j=1}^n \Big(\frac{\partial^2}{\partial x_j^2}+\frac{2k_j}{x_j}\frac{\partial}{\partial x_j}\Big).$$

\begin{proposition}\label{MaximumPrinciple}
 Let $f(x_0,x_1,...,x_n)$ be a $C^2$ function defined on an open connected set
 $\Omega \subset (0,\infty)\times \mathbb R^n$.
 Assume that $f(x_0,x_1,...,-x_j,...,x_n)= f(x_0,x_1,...,x_j,...,x_n)$ whenever
 $(x_0,x_1,...,-x_j,...,x_n)$ and $(x_0,x_1,...,x_j,...,x_n)$ belong to $\Omega$
 and  $\widetilde{\mathbb L} f\geq 0$ on the set
 $\{(x_0,...,x_n)\in \Omega: x_1\cdot x_2\cdot ...\cdot x_n\ne 0\}$.
 Then $f$ cannot attain a local maximum in $\Omega$ unless $f$ is a constant.
\end{proposition}
 \begin{proof} The proposition is a corollary of Theorem 4.2 of \cite{Roesler2}. For the convenience of the reader  we present here an alternative proof based one  ideas from \cite{MucStein}, where the one dimensional Bessel operator was considered.

  Denote $\boldsymbol x=(x_0,\mathbf x)=(x_0,x_1,...,x_n)\in (0,\infty)\times \mathbb R^n$,  $U=\{\boldsymbol x\in (0,\infty)\times \mathbb R^n: x_1\cdot x_2\cdot ...\cdot x_n\ne 0\}$. Set
 $$v(x_0,x_1,...,x_n)
 =|x_1|^{2k_1}|x_2|^{2k_2}\cdot ...\cdot |x_n|^{2k_n}.$$
 By the divergence theorem
  for $C^2$ functions $f$ and   $g$ in a smooth region $ \bar D$ one has
  \begin{equation}\label{eq3.55}
   \int_D [g\, \text{div}(v\nabla f) - f\, \text{div}(v\nabla g)]\, dx_0\, dx_1...dx_n=
   \int_{\partial D} v\Big(g\frac{\partial f}{\partial \boldsymbol n}-f\frac{\partial g}{\partial \boldsymbol n}\Big)\, ds,
  \end{equation}
  where $\boldsymbol n$ is outward normal vector to $D$ at $\boldsymbol x\in \partial D$. Observe that
  $ \text{div}(v\nabla f)=v\,\widetilde{\mathbb L}f$ on $D\cap U$.
  So, if $g$ is additionally  $\widetilde{\mathbb L}$-harmonic on  $ D\cap U$ then $\text{div} (v\nabla g)=v\widetilde{\mathbb L} g=0$ on $D\cap U$, and
  setting $f\equiv 1$ in (\ref{eq3.55}) we get
  \begin{equation}\label{eq3.7}
   \int_{\partial D}v \frac{\partial g}{\partial \boldsymbol n}\, ds =0.
  \end{equation}

  Assume that at    $\boldsymbol a=(a_0,a_1,...,a_n)\in \Omega$ the function $f$ attains a local maximum.
  By Hopf's maximum principle (see \cite[Section 6.4.2, Theorem 3]{Evans})  $\boldsymbol a$ is not a regular point of $\widetilde{\mathbb L}$, that is,
  $a_1\cdot a_2\cdot...\cdot a_n=0$.
  There is no loss of generality in assuming that there is  $m\in\{1,2,...,n\}$ such that  $a_0>0,..., a_{m-1}>0$,
  $a_{m}=a_{m+1}=...=a_n=0$.
 Let
 $$h^{[0]}_\tau(x_0,a_0)=\frac{1}{\sqrt{4\pi \tau}}
 \exp (-|x_0-a_0|^2\slash 4\tau),$$
 $$h^{[j]}_\tau (x_j,0)=\frac{1}{\Gamma (\lambda_j+1\slash 2)} \tau^{-k_j-1\slash 2}
 \exp(-x^2_j\slash 4\tau ), \ \ j\in \{m,m+1,...,n\},$$
 $$h^{[j]}_\tau (x_j,a_j)={\tt h}_\tau^{[j]}(x_j,a_j)  \ \ \text{for }
 j\notin \{0,m,m+1,..., n\}$$
 (see (\ref{heatBessel})). Put
 \begin{equation*}\begin{split}
  g_0(x_0,x_1,...,x_n)
 =\int_0^\infty \prod_{j=0}^{n} h_\tau ^{[j]}(x_j,a_j)\, d\tau .
\end{split} \end{equation*}
 We have  $\widetilde{\mathbb L} g_0=0$ on $\big((0,\infty)^{m}\times \mathbb R^{n+1-m}\big)\cap U$.
It is not difficult to check using the asymptotic behavior of the Bessel functions
$I_{\nu}$ (see, e.g., \cite{NIST}) that there is $r>0$
 such that $\nabla g_0 (\boldsymbol x)\ne 0$ for every $\boldsymbol x\in B(\boldsymbol a,2r)\setminus \{ \boldsymbol a\}\subset D$.
 Let  $D_R=\{ \boldsymbol x: g_0(\boldsymbol x)>R\}\cup\{ \boldsymbol a\}$. We take  $R$ large enough such that  $D_R\subset B(\boldsymbol a,r)$. For  $\varepsilon >0$ small enough let $D_{R,\varepsilon}=
 D_R\setminus B(\boldsymbol a,\varepsilon)$. Set $g(\boldsymbol x)=g_0(\boldsymbol x)-R$. Then $g\equiv 0$ on $\partial D_R$, $g\geq 0$ on
 $D_{R,\varepsilon}$ and $\frac{\partial}{\partial \boldsymbol n} g(\boldsymbol x)<0$ on $\partial D_R$, where $\boldsymbol n$ is outward  normal
 vector to
 $D_R$ at $\boldsymbol x \in \partial D_R$. Using (\ref{eq3.7}) we have
 \begin{equation}
   \int_{\partial D_R}v \frac{\partial g}{\partial \boldsymbol n}\, ds =
   \int_{\partial B(a,\varepsilon)}v \frac{\partial g}{\partial \boldsymbol n}\, ds <0.
  \end{equation}
 Now from (\ref{eq3.55}) we conclude that
 \begin{equation}\begin{split}
 0 & \leq \int_{D_{R,\varepsilon}} g(v\widetilde{\mathbb L} f) \, dx_0dx_1...dx_n\\
 &= \int_{D_{R,\varepsilon}} g\, \text{div }(v\nabla f) \, dx_0dx_1...dx_n\\
 & =-\int_{\partial D_R}  fv\frac{\partial g}{\partial \boldsymbol n}ds
 -\int_{\partial B(\boldsymbol a,\varepsilon)}g v \frac{\partial f}{\partial \boldsymbol n}\, ds +
 \int_{\partial B(\boldsymbol a,\varepsilon)}f v \frac{\partial g}{\partial \boldsymbol n}\, ds,
\end{split} \end{equation}
where in the last two integrals $\boldsymbol n$ is the outward normal vector to $B(\boldsymbol a,\varepsilon)$.
 Clearly, the second summand tends to 0 as $\varepsilon $ tends to 0. On the other hand, by (\ref{eq3.7}),  the third summand tends to
 $f(\boldsymbol a) \int_{\partial D_R}  v\frac{\partial g}{\partial \boldsymbol n}ds$.
 Thus,
 \begin{equation}\label{eq3.45}
  0\leq \int_{\partial D_R} (f(\boldsymbol a)- f)v\frac{\partial g}{\partial \boldsymbol n}ds.
 \end{equation}
 Recall that $f$ attains  a local maximum at $\boldsymbol a$ and  $\frac{\partial g}{\partial\boldsymbol n}<0$ on $\partial D_R$.
 Hence, from (\ref{eq3.45}) we deduce that  $f=f(\boldsymbol a)$ on $\partial D_R$.
 So  $f$ must be a constant in  a neighborhood of $\boldsymbol a$ and, consequently, $f\equiv f(\boldsymbol a)$ on $\Omega$,
 since $\Omega$ is connected.
 \end{proof}

\section{Proof of Theorem \ref{Theorem1} }

\begin{proof}[Proof of Theorem \ref{Theorem1}] The second inequality in (\ref{eq222}) is a direct consequence of the following multiplier theorem (see \cite[Theorem 1.10]{ABDH}).

\begin{theorem}\label{Theorem2}
Let \,$\chi\ssb=\ssb\chi(\boldsymbol{\xi})$
be a smooth radial function on \,$\R^n$ such that
\begin{equation*}
\chi(\boldsymbol{\xi})=\begin{cases}
\,1&\text{if \;}|\boldsymbol{\xi}|\!\in\!\bigl[\frac12,2\ssf\bigr]\ssf,\\
\,0&\text {if \;}|\boldsymbol{\xi}|\!\notin\!\bigl(\frac14,4\ssf\bigr)\ssf.\\
\end{cases}\end{equation*}
If a function \;$m\ssb=\ssb m(\boldsymbol{\xi})$ \ssf on \,$\R^n$ satisfies
\begin{equation}\label{multiplier222}
M=\,\sup\nolimits_{\,t>0}\,
\|\,\chi\,m(t\,.\,)\ssf\|_{\hspace{.1mm}W_{\ssf2}^{\ssf\mathbf{N}/2\ssf+\ssf\epsilon}}
<+\infty\,,
\end{equation}
for some \,$\epsilon\!>\!0$\ssf,
then the multiplier operator
\begin{equation*}
\mathcal{T}_{\ssf m\ssf}f
=\mathcal{F}^{-1}\{\ssf m\,(\mathcal{F}\ssb f)\}
\end{equation*}
is bounded on the Hardy space \ssf$H^1_{{\rm max}, \mathbf L}$ and
\begin{equation*}
\|\,\mathcal{T}_{\ssf m}\,\|_{H^1_{{\rm max}, \mathbf L}\to\ssf H^1_{{\rm max}, \mathbf L}}\lesssim\,M\ssf.
\end{equation*}
\end{theorem}
It is not difficult to check that the multiplier $m_j(\boldsymbol \xi)=i \frac{\xi_j}{|\boldsymbol \xi| } $, which corresponds to the Riesz transform $\mathcal R_j$, satisfies (\ref{multiplier222}). Hence $\mathcal R_j$ is bounded from $H^1_{{\rm max}, \mathbf L}$ to itself and, consequently,
form $H^1_{{\rm max}, \mathbf L}$ to $L^1(\mathbb R^n,\, d\boldsymbol \mu)$.

 Now we turn to prove the first inequality in (\ref{eq222}). For this purpose we  use Theorem  \ref{propositionP}, Propositions \ref{propositionP2}, \ref{subharmonic},  and \ref{MaximumPrinciple} combined  with the steps of  the proof of the characterization of  the classical Hardy spaces by the classical Riesz transforms (see, e.g., \cite[Chapter III, Section 4]{Stein}). For the convenience of the reader, we provide the details.

 Assume that $f\in L^1(\mathbb R^n, \, d\boldsymbol\mu )$ and $\mathcal R_jf \in L^1(\mathbb R^n, \, d \boldsymbol\mu )$ for $j=1,2,...,n$.
There is no loss of generality in assuming that $f$ is real valued, and hence so are $\mathcal R_jf$.  Set $\mathbf u (x_0,x_1,...,x_n)=(u_0,u_1,...,u_n)$, where $u_j$ are defined in (\ref{vectoru}) (recall that $x_0=t>0$). Fix $0<q<1$ as in Proposition \ref{subharmonic} and set $p=1\slash q$. Let
$F(x_0,\mathbf x)=\{\mathbf u^\sigma(x_0,\mathbf x)\}_{\sigma\in \mathcal G}$.
Clearly,
 \begin{equation}
  \sup_{x_0>0} \int_{\mathbb R^n} |F(x_0, x_1,...,x_n)|d\boldsymbol \mu (x_1,...,x_n)
  \leq C\Big(\| f\|_{L^1(\mathbb R^n,\, d\boldsymbol\mu )}+\sum_{j=1}^n \| \mathcal R_jf\|_{L^1(\mathbb R^n,\, d\boldsymbol \mu)}\Big).
 \end{equation}
 Denote $F_\varepsilon (x_1,...,x_n)= F(\varepsilon, x_1,...,x_n)$. Then
 $|F_\varepsilon|\in C_0(\mathbb R^n)$ (see part (b) of Proposition \ref{propositionP2}) and, by (\ref{App5}),
 \begin{equation}\label{eq333} \sup_{\varepsilon >0} \| |F_\varepsilon |^q\|^p_{L^p(\mathbb R^n,\,  d\boldsymbol \mu )}\leq C
 \Big(\| f\|_{L^1(\mathbb R^n,\, d\boldsymbol\mu)}+\sum_{j=1}^n \| \mathcal R_jf\|_{L^1(\mathbb R^n, \, d \boldsymbol\mu )}\Big).
 \end{equation}
 Consider
 the function $G:[0,\infty)\times \mathbb R^n \to \mathbb R$,
 $$G(x_0,x_1,...,x_n)= |F(\varepsilon + x_0,x_1,...,x_n)|^q -P_{x_0}
 (|F_\varepsilon| ^q|)(x_1,...,x_n).$$
  The function is continuous vanishes for $x_0=0$ and, by Proposition \ref{propositionP2},
  $$\lim_{(x_0+|(x_1,...,x_n)|)\to\infty}  G(x_0,x_1,...,x_n)=0.$$
 Moreover, $G(x_0,\mathbf x)=G(x_0,\sigma\mathbf x)$ for every $\sigma\in\mathcal G$.
 We claim  that
 \begin{equation}\label{eq334}
 G(x_0,\mathbf x)=|F(\varepsilon+x_0,\mathbf x)|^q -
 P_{x_0}
 (|F_\varepsilon| ^q)(\mathbf x)\leq 0.
 \end{equation}
  To prove the claim assume that $G>0$ at some point. Then it attains a global maximum, say at $\boldsymbol a=(a_0,a_1,...,a_n)$. Obviously,  $a_0>0$ and $|F(\boldsymbol a)|>0$.
  Take a connected neighborhood $\Omega$ of $\boldsymbol a$ such
 that $G>0$ on $\Omega$ and $G$ is not constant on $\Omega$.  Then $G$ is $C^2$ on $\Omega$ and, according to (\ref{harmonic}) and Proposition \ref{subharmonic}, $\mathcal LG=\widetilde{\mathbb L}G=\widetilde{\mathbb L} |F|^q\geq 0$ on
 $\{ (x_0,x_1,...,x_n)\in \Omega : \ x_1\cdot x_2\cdot ...\cdot x_n\ne 0\}$.
 This contradicts  the maximum principle (see Proposition \ref{MaximumPrinciple}). Hence (\ref{eq334}) is proved.

 It follows from (\ref{eq333}) that there is a sequence $\varepsilon_n\to 0$ such that $|F_{\varepsilon_n}|^q$
 converges in a
 weak * topology  of the Banach space $L^p(\mathbb R^n, d\boldsymbol \mu)$ to $h\in L^p(\mathbb R^n, d\boldsymbol \mu ) $ and
 \begin{equation}\label{functionh}\| h\|^p_{L^p(\mathbb R^n,\, d\boldsymbol\mu) }\leq C \Big(\| f\|_{L^1(\mathbb R^n, \, d\boldsymbol\mu ))}+\sum_{j=1}^n \| \mathcal R_jf\|_{L^1(\mathbb R^n, \, d\boldsymbol\mu ))}\Big).
 \end{equation}
 From (\ref{eq334}) and (\ref{Lp}) we conclude that
 \begin{equation}\label{eq5.12} |F(x_0,\mathbf x )|^q \leq P_{x_0} h(\mathbf x).
 \end{equation}
  Since the maximal function
 $P_*$ is bounded on $L^p(\mathbb R^n, \, d\boldsymbol\mu (x))$ (see Theorem \ref{propositionP}), we deduce from (\ref{eq5.12}) and (\ref{functionh})
 that
 \begin{equation*}\begin{split}
\int  \sup_{x_0>0} |u_0(x_0,\mathbf x)|\, d\boldsymbol \mu(\mathbf x)
 &\leq C \int \Big( \sup_{x_0>0} P_{x_0} h(\mathbf x)\Big)^p
 d\boldsymbol \mu(\mathbf x)\\
 &\leq C\| h\|_{L^p(\mathbb R^n, \, d \boldsymbol\mu )}^p\\
 &\leq C (\| f\|_ {L^1(\mathbb R^n, \,d \boldsymbol\mu)}  +\sum_{j=1}^n \| \mathcal R_jf\|_{L^1(\mathbb R^n, \, d \boldsymbol\mu)}).
 \end{split} \end{equation*}
  Finally, from Theorem \ref{propositionP} we get $f\in H^1_{\rm{max},\mathbf L}$ and
  $$\| f\|_{H^1_{\rm{max},\mathbf L}}\leq  C (\| f\|_{L^1(\mathbb R^n, \, d \boldsymbol\mu)} +\sum_{j=1}^n \| \mathcal R_jf\|_{L^1(\mathbb R^n, \, d \boldsymbol\mu)}).$$
\end{proof}
 \

 \section{Proof of Theorem \ref{TheoremBessel}}

\begin{proof}[Proof of Theorem \ref{TheoremBessel}] Recall   that
\begin{equation}\label{App1} \mathbf h_{t}(\mathbf x,\mathbf y)=\prod_{j=1}^n h^{\{j\}}_t (x_j,y_j)
\end{equation}
(see (\ref{DunklKernelProduct}) and (\ref{DunklKernel1D})), where
$$h^{\{j\}}_t (x,y)=(4t)^{-1} \exp(-(x^2+y^2)\slash 4t)|xy|^{-k_j+1\slash 2}\Big(I_{k_j-1\slash 2}\Big(\frac{|xy|}{2t}\Big)+{\rm sgn}(xy)I_{k+1\slash 2}\Big(\frac{|xy|}{2t}\Big)\Big)$$ is the heat kernel associated with one dimensional Dunkl operator
$$L f(x)=\ssb f''(x)\ssb
+\ssb\tfrac{2\ssf k_j}{x}f'(x)\ssb
-\ssb\tfrac{k_j}{x^2}\ssf\bigl(\ssf f(x)\!
-\!f(-x)\ssf\bigr).$$
Clearly, $h^{\{j\}}_{\ssf t}(x,y)=h^{\{j\}}_{\ssf t}(y,x)$ is a $C^\infty$ function of $(t,x,y)$.
For a function $f$ defined on $(0,\infty)^n$ let $\widetilde f$ denote its extention to the  $\mathcal G$ invariant function on $\mathbb R^n$. One can easily check using (\ref{productBessel}), (\ref{heatBessel}) that
$$ (e^{t\mathbb L} f)\widetilde \ =e^{t\mathbf L}(\widetilde f).$$
Hence, $f$ belongs to the Hardy space $H^1_{\rm{max},\mathbb L}$ if and only if $\widetilde f\in H^1_{{\rm max}, \mathbf L}$\,. Moreover, $\| f\|_{H^1_{\rm{max},\mathbb L}}=c\| \widetilde f\|_{H^1_{{\rm max}, \mathbf L}}$. Let us note that
$$|\mathcal R_j \widetilde f|=|(R_j f) \widetilde \ \, |.$$
 Thus Theorem \ref{TheoremBessel} follows from Theorem \ref{Theorem1}.
\end{proof}

\section{Appendix}

It is well known that
\begin{equation}\label{App2}
 \int_{\mathbb R} h^{\{j\}}_{\ssf t}(x,y)\, d\mu_j(y)=1, \ \ \ d\mu_j(y)= |y|^{2k_j}\, dy.
 \end{equation}
It was proved in \cite{ABDH} that   $h^{\{j\}}_t (x,y)$
 has the following global behavior\,{\rm:}
\begin{equation}\label{heatB}
h^{\{j\}}_{\ssf t}(x,y)\,\asymp\,\begin{cases}
\;t^{-k_j-\frac12}\,
e^{-\frac{x^2\ssb+\ssf y^2}{4\ssf t}}
&\text{if \;}|\ssf x\ssf y\ssf|\!\le\ssb t\ssf,\\
\;t^{-\frac12}\,(x\ssf y)^{-k_j}\,
e^{-\frac{(x-y)^2}{4\ssf t}}
&\text{if \;}x\ssf y\ssb\ge\ssb t\ssf,\\
\;t^{\ssf\frac12}\,(-\ssf x\ssf y)^{-k_j-1}\,
e^{-\frac{(x+y)^2}{4\ssf t}}
&\text{if \,}-\ssb x\ssf y\ssb\ge\ssb t\ssf.\\
\end{cases}
\end{equation}
From (\ref{heatB}) we easily conclude that
\begin{equation}\label{App3}
0<\mathbf h_{t}(\mathbf x,\mathbf y)\leq \frac{C}{\boldsymbol \mu(B(\mathbf x,\sqrt{t}))} \ \ \text{for all } \ \mathbf x, \mathbf y \in\mathbb R^n \ \text{and} \ t>0;
  \end{equation}
   \begin{equation}\label{App4} \mathbf h_{t}(\mathbf x,\mathbf y)\leq \frac{C}{\boldsymbol \mu(B(\mathbf x,\sqrt{t}))}
   e^{-c|\mathbf x|^2\slash t}\ \ \ \text{ for } \  |\mathbf x|>2n|\mathbf y|. \\
\end{equation}
We shall need the following inequalities for volumes of the Euclidean balls (see \cite{ABDH})
\begin{equation}\label{ComparisonVolumeBalls}
\bigl(\tfrac Rr\bigr)^{\ssb n}\ssb
\lesssim\tfrac{\boldsymbol{\mu}\ssf(\ssf\mathbf{B}(\mathbf{x},\ssf R\ssf))}
{\boldsymbol{\mu}\ssf(\ssf\mathbf{B}(\mathbf{x},\ssf r))}
\lesssim\ssb\bigl(\tfrac Rr\bigr)^{\ssb\mathbf{N}},
\qquad\forall\;\mathbf{x}\!\in\!\R^n,\,\forall\;R\ssb\ge\ssb r\!>\ssb0\ssf.
\end{equation}
The subordination formula (\ref{subordination}) combined with (\ref{App1}) and (\ref{App2}) implies
\begin{equation}\label{App5}
\int_{\mathbb R^n} P_t(\mathbf x,\mathbf y)\, d\boldsymbol\mu (\mathbf x)=\int_{\mathbb R^n} P_t(\mathbf x,\mathbf y)\, d\boldsymbol\mu (\mathbf y)=1.
\end{equation}
\begin{lemma}
There is a constant $C>0$ such that
\begin{equation}\label{App6}
0< P_t(\mathbf x,\mathbf y)\leq \frac{C}{\boldsymbol \mu (B(\mathbf x, t))}.
\end{equation}
Moreover, for every $0<\delta <\frac{1}{\mathbf N}$ there is a constant $C_\delta$ such that \begin{equation}\label{App7}
 P_t(\mathbf x,\mathbf y)\leq \frac{C_\delta}{\boldsymbol \mu (B(\mathbf x, t))}\Big(1+\frac{ \boldsymbol \mu (B(\mathbf x, |\mathbf x|))}{\boldsymbol\mu (B(\mathbf x,t))}\Big)^{-1-\delta} \ \ \text{for } \ |\mathbf x|>2n|\mathbf y|.
\end{equation}
\end{lemma}
\begin{proof} To see (\ref{App6}) we use (\ref{subordination}) together with (\ref{App3}) and (\ref{ComparisonVolumeBalls}) and obtain
\begin{equation*}\begin{split}
 P_t(\mathbf x,\mathbf y) &\lesssim \int_0^{\frac{1}{4}} \frac{1}{\boldsymbol \mu (B(\mathbf x,t))}
 \frac{\boldsymbol \mu (B(\mathbf x,t))}{\boldsymbol \mu (B(\mathbf x,\frac{t}{2\sqrt{u}}))}\frac{du}{\sqrt{u}}+ \int_{\frac{1}{4}}^\infty  \frac{e^{-u}}{\boldsymbol \mu (B(\mathbf x,t))}
 \frac{\boldsymbol \mu (B(\mathbf x,t))}{\boldsymbol \mu (B(\mathbf x,\frac{t}{2\sqrt{u}}))}\frac{du}{\sqrt{u}}\\
 &\lesssim \frac{1}{\mu (B(\mathbf x,t))} \Big(\int_0^{\frac{1}{4}} u^{n\slash 2}\frac{d}{\sqrt{u}}+\int_{\frac{1}{4}}^\infty e^{-u} u^{\mathbf N\slash 2}\frac{du}{\sqrt{u}}\Big)
 \lesssim  \frac{1}{\mu (B(\mathbf x,t))}.
\end{split}\end{equation*}
The proof of the lower bound of $P_t(\mathbf x,\mathbf y)$ is obvious.

In order to prove (\ref{App7}) it suffices to consider $t\leq |\mathbf x|\slash 2$.
By (\ref{ComparisonVolumeBalls}), for every $\delta>0$ and $c>0$, we have
\begin{equation}\label{App10}
 \Big(1+\frac{\boldsymbol\mu (B(\mathbf x, |\mathbf x|))}{\boldsymbol\mu (B(\mathbf x,\sqrt{s}))} \Big)^{1+\delta}\leq C_\delta\Big(1+\frac{|\mathbf x|}{\sqrt{s}}\Big)^{(1+\delta) \mathbf N} \leq
C_{\delta, c} e^{c|\mathbf x|^2\slash s}, \ \ \text{for } \ s>0.
\end{equation}
Utilizing (\ref{App4}) together with (\ref{App10}) and proceeding similarly to the proof of (\ref{App6}) we have
\begin{equation}\begin{split}\nonumber
 P_t(\mathbf x,\mathbf y) & \lesssim \int_0^\infty \frac{e^{-u}}{\boldsymbol \mu (B(\mathbf x, \frac{t}{2\sqrt{u}}))} \Big(1+\frac{\boldsymbol \mu (B(\mathbf x, |\mathbf x|))}{\boldsymbol\mu (B(\mathbf x, \frac{t}{2\sqrt{u}}))}\Big)^{-1-\delta}\frac{du}{\sqrt{u}}\\
 &= \int_0^{\infty} \frac{e^{-u}}{\boldsymbol \mu (B(\mathbf x,t))}
 \frac{\boldsymbol \mu (B(\mathbf x,t))}{\boldsymbol \mu (B(\mathbf x,\frac{t}{2\sqrt{u}}))} \Big(1+\frac{\boldsymbol \mu (B(\mathbf x, |\mathbf x|))}{\boldsymbol \mu (B(\mathbf x, t))} \frac{\boldsymbol\mu (B(\mathbf x, t))}{\boldsymbol\mu (B(\mathbf x, \frac{t}{2\sqrt{u}}))}\Big)^{-1-\delta} \frac{du}{\sqrt{u}}\\
 &\lesssim \int_0^{\frac{1}{4}} \frac{1}{\boldsymbol \mu (B(\mathbf x,t))}
 \frac{\boldsymbol \mu (B(\mathbf x,t))}{\boldsymbol \mu (B(\mathbf x,\frac{t}{2\sqrt{u}}))} \Big(1+\frac{\boldsymbol \mu (B(\mathbf x, |\mathbf x|))}{\boldsymbol \mu (B(\mathbf x, t))}\Big)^{-1-\delta} \Big( \frac{\boldsymbol\mu (B(\mathbf x, t))}{\boldsymbol\mu (B(\mathbf x, \frac{t}{2\sqrt{u}}))}\Big)^{-1-\delta} \frac{du}{\sqrt{u}}\\
 &
 + \int_{\frac{1}{4}}^{\infty} \frac{e^{-u}}{\boldsymbol \mu (B(\mathbf x,t))}
 \frac{\boldsymbol \mu (B(\mathbf x,t))}{\boldsymbol \mu (B(\mathbf x,\frac{t}{2\sqrt{u}}))} \Big(1+\frac{\boldsymbol \mu (B(\mathbf x, |\mathbf x|))}{\boldsymbol \mu (B(\mathbf x, t))}\Big)^{-1-\delta} \frac{du}{\sqrt{u}}.\\
\end{split}\end{equation}
Fix $0<\delta<\mathbf N^{-1}$. Applying (\ref{ComparisonVolumeBalls}) we obtain
\begin{equation}\begin{split}\nonumber
 P_t(\mathbf x,\mathbf y) &\lesssim \int_0^{\frac{1}{4}} \frac{1}{\boldsymbol \mu (B(\mathbf x,t))} \Big(1+\frac{\boldsymbol \mu (B(\mathbf x, |\mathbf x|))}{\boldsymbol \mu (B(\mathbf x, t))}\Big)^{-1-\delta} u^{-\delta\mathbf N\slash 2} \frac{du}{\sqrt{u}}\\
 &
 + \int_{\frac{1}{4}}^{\infty} \frac{e^{-u}u^{\mathbf N\slash 2}}{\boldsymbol \mu (B(\mathbf x,t))}
 \Big(1+\frac{\boldsymbol \mu (B(\mathbf x, |\mathbf x|))}{\boldsymbol \mu (B(\mathbf x, t))}\Big)^{-1-\delta} \frac{du}{\sqrt{u}},\\
\end{split}\end{equation}
which proves (\ref{App7}).
\end{proof}

 \begin{proof}[Proof of Theorem \ref{propositionP}] The  proof, which is its spirit similar to that of the heat kernel characterization of $H^1_{\rm{atom}}$ (see \cite{ABDH}), is  based on  the following result due to Uchiyama \cite{Uchiyama}.
\begin{theorem}\label{TheoremUchiyama}
Assume that a set \,$X$ is equipped with
\begin{itemize}
\item[$\bullet$]
a \ssf{\rm quasi-distance} \ssf$\widetilde{d}$
\ssf i.e.~a distance except that
the triangular inequality is replaced by the weaker condition
\vspace{.5mm}

\centerline{$
\widetilde{d}\ssf(x,y)\le A\,\{\ssf\widetilde{d}\ssf(x,z)+\widetilde{d}\ssf(z,y)\ssf\},
\qquad\forall\;x,y,z\!\in\!X\ssf;
$}\vspace{.5mm}

\item[$\bullet$]
a measure \,$\mu$ whose values on quasi-balls satisfy
\vspace{.5mm}

\centerline{$
\frac rA\le\mu\ssf(\widetilde{B}\ssf(x,r))\le r,
\qquad\forall\;x\!\in\!X\ssf,\,\forall\;r\!>\!0\,;
$}\vspace{1mm}

\item[$\bullet$]
a continuous kernel \,$K_r(x,y)\!\ge\!0$ such that,
for every \,$r\!>\!0$ and \,$x,y,y^{\ssf\prime}\hspace{-1mm}\in\!X$,
\vspace{1mm}

\begin{itemize}
\item[$\circ$]
$K_r(x,x)\ge\frac1{A\hspace{.5mm}r}$\,,

\item[$\circ$]
$K_r(x,y)\le r^{-1}\ssf\bigl(\ssf
1\ssb+\ssb\frac{\widetilde{d}\ssf(x,\ssf y)}r
\ssf\bigr)^{\ssb-1-\ssf\delta}$\,,

\item[$\circ$]
$\bigl|\ssf K_r(x,y)\ssb-\ssb K_r(x,y^{\ssf\prime})\bigr|\ssb
\le\ssb r^{-1}\ssf\bigl(\ssf1\ssb
+\ssb\frac{\widetilde{d}\ssf(x,\ssf y)}r\ssf\bigr)^{\ssb-1-\ssf2\ssf\delta}\ssf
\bigl(\ssf\frac{\widetilde{d}\ssf(y,\ssf y^{\ssf\prime})}r\ssf\bigr)^{\ssb\delta}$
\ssf when
\,$\widetilde{d}\ssf(y,y^{\ssf\prime})\!
\le\!\frac{r\ssf+\ssf\widetilde{d}\ssf(x,\ssf y)}{4\ssf A}$\,.
\vspace{.5mm}

\end{itemize}
\end{itemize}
Here \,$A\!\ge\!1$ and \,$\delta\!>\!0$\ssf.
Then the following definitions of the Hardy space \,$H^1(X)$ and their corresponding norms
are equivalent\,{\rm:}
\begin{itemize}
\item[$\bullet$]
{\rm Maximal definition\,:}
$H^1_{{\rm max}, K_r}(X)$ \ssf consists of all functions \,$f\!\in\!L^1(X,d\mu)$ such that
\begin{equation*}
K_*f(x)=\ssf\sup\nolimits_{\,r>0}\,
\Bigl|\ssf{\displaystyle\int_X}K_r(x,y)\,f(y)\,d\mu(y)\Bigr|
\end{equation*}
belongs to \ssf$L^1(X,d\mu)$
and the norm \,$\|f\|_{H^1_{{\rm max}, K_r}(X)} = \|K_*f\|_{L^1(X,d\mu)}$.
\item[$\bullet$]
{\rm Atomic definition\,:} An atom for $H^1_{\rm atom}(X,\widetilde d)$  is a measurable function \,$a\ssb:\ssb X\!\to\ssb\C$
\,such that: \ssf$a$ \ssf is supported in a quasi-ball \ssf$\widetilde{B}$\ssf,
\,$\|a\|_{L^\infty}\!\lesssim\ssb\mu\ssf(\widetilde{B})^{-1}$
\ssf and \,$\displaystyle\int_Xa \,d\mu\ssb=\ssb0$\ssf (see, \cite{CoifmanWeiss, MaciasSegovia, Uchiyama}).
Then
$H^1_{\rm atom}(X, \widetilde d)$ \ssf consists of all functions \,$f\!\in\!L^1(X,d\mu)$
which can be written as
\,$f\!=\!\sum_{\ssf\ell}\ssb\lambda_{\ssf\ell}\ssf a_{\ssf\ell}$\ssf,
where the \ssf$a_{\ssf\ell}$\ssb's are atoms
and \,$\sum_{\ssf\ell}\ssb|\lambda_{\ssf\ell}|\!<\!+\infty$\ssf,
and the norm norm $\| f\|_{H^1_{\rm atom}(X,\widetilde d)}= \inf \sum_{\ssf\ell}\ssb|\lambda_{\ssf\ell}|$
\ssf over all such representations.
\end{itemize}
\end{theorem}

For \ssf$X\!=\ssb\R^n$,
equipped with the Euclidean distance
\,$d\ssf(\mathbf{x},\mathbf{y})\ssb=\ssb|\ssf\mathbf{x}\ssb-\ssb\mathbf{y}\ssf|$
\ssf and the measure $\boldsymbol \mu$ (see \eqref{ProductMeasure}),
set
\begin{equation*}
\widetilde{d}\ssf(\mathbf{x},\mathbf{y})=\inf\boldsymbol{\mu}\ssf(B),
\qquad\forall\;\mathbf{x},\mathbf{y}\!\in\!\R^n,
\end{equation*}
where the infimum is taken over all closed balls \ssf$B$
\ssf containing \ssf$\mathbf{x}$ \ssf and \ssf$\mathbf{y}$. Let $t\ssb=\ssb t\ssf(\mathbf{x},r)$ be defined by
\ssf$\boldsymbol{\mu}\ssf(B\ssf(\mathbf{x},\ssb\sqrt{t\ssf}\ssf))\ssb=\ssb r$.
Then
$$ \boldsymbol \mu(\widetilde B(\mathbf x,r))\sim r$$
and there exists a constant \ssf$c\!>\!0$ \ssf such that
\begin{equation}\label{balls}
B(\mathbf{x},\ssb\sqrt{\ssf t\,})
\subset\widetilde{B}(\mathbf{x},r)
\subset B(\mathbf{x},c\hspace{.4mm}\sqrt{\ssf t\,}),
\end{equation}
where $\widetilde{B}(\mathbf{x},r)=\{ \mathbf y\in \mathbb R^n: \widetilde d(\mathbf x,\mathbf y)<r\}$
(see, e.g., \cite{ABDH}).

Let us remark that thanks to (\ref{balls}) and (\ref{ComparisonVolumeBalls}) the atomic spaces $H^1_{\rm atom}(X,\widetilde d)$ and $H^1_{\rm atom}$ (defined in Section 1) do coincide and $\| f\|_{H^1_{\rm atom}(X,\widetilde d)} \sim \| f\|_{H^1_{\rm atom}}$.

It was proved in \cite{ABDH} that the kernel $\mathbf h_t$ can be written in the form
\begin{equation*}
{\mathbf{h}}_{\ssf t}(\mathbf{x},{\mathbf{y}})
={\mathbf{H}}_{\ssf t}(\mathbf{x},{\mathbf{y}})
+{\mathbf{S}}_{\ssb t}(\mathbf{x},{\mathbf{y}})\ssf,
\end{equation*}
where ${\mathbf{H}}_{\ssf t}(\mathbf{x},{\mathbf{y}})$
and ${\mathbf{ S}}_{\ssb t}(\mathbf{x},{\mathbf{y}})\ssf$ are nonnegative functions such that
there are $C_1,\, C_2\, C_4,\, \delta>0$ such that

\begin{equation}\label{OnDiagonalLowerEstimateKr}
\mathbf H_t(\mathbf{x},\mathbf{x})\ge\tfrac{C_1}{\boldsymbol \mu (B(\mathbf x,\sqrt{t}))};
\end{equation}
\begin{equation}\label{UpperEstimateKr}
\mathbf H_t(\mathbf{x},\mathbf{y})\le\tfrac{C_2}{ \boldsymbol \mu (B(\mathbf x,\sqrt{t}))} \,\bigl(\ssf1\ssb
+\ssb\tfrac{\widetilde{d}\ssf(\mathbf{x},\ssf\mathbf{y})} { \boldsymbol \mu (B(\mathbf x,\sqrt{t}))} \ssf\bigr)^{\ssb-1-\delta};
\end{equation}

\begin{equation}\label{LipschitzEstimateKr}
\bigl|\ssf \mathbf H_t (\mathbf{x},\mathbf{y})\ssb
-\ssb \mathbf H_t(\mathbf{x},\mathbf{y}^{\ssf\prime})\ssf\bigr|
\le\tfrac{C_4}{\boldsymbol \mu ( B (\mathbf x,\sqrt{t}))} \ssf\bigl(\ssf
1\ssb+\ssb\tfrac{\widetilde{d}\ssf(\mathbf{x},\ssf\mathbf{y})}{\boldsymbol \mu (B(\mathbf x,\sqrt{t}))}
\ssf\bigr)^{\ssb-1-2\ssf\delta}\ssf
\bigl(\ssf\tfrac{\widetilde{d}\ssf(\mathbf{y},\ssf\mathbf{y}^{\ssf\prime})}{\boldsymbol \mu ( B(\mathbf x,\sqrt{t}))}
\ssf\bigr)^{\ssb\frac1{\mathbf{N}}}
\end{equation}
for $\widetilde{d}\ssf(\mathbf{y},\mathbf{y}^{\ssf\prime})\ssb
\le\ssb C_3\max\,\{\ssf \boldsymbol \mu (B(\mathbf x,\sqrt{t})),\widetilde{d}\ssf(\mathbf{x},\mathbf{y})\ssf\}$\ssf,
 (the kernel $\mathbf S_t$ is denoted in \cite{ABDH} by $\mathbf P_t$).
Moreover, the maximal function
$$ \mathbf S_*f(x)=\sup_{t>0} \Big|\int \mathbf S_t(\mathbf x,\mathbf y)f(\mathbf y)d\boldsymbol \mu(\mathbf y) \Big| $$
is a bounded operator on $L^1(\mathbb R^n,\, d\boldsymbol\mu )$.

Using  subordination formula (\ref{subordination})  we write
\begin{equation}\label{splitP}
 P_t(\mathbf x,\mathbf y) = U_t(\mathbf x,\mathbf y) + W_t(\mathbf x,\mathbf y),
 \end{equation}
where
$$ U_t(\mathbf x,\mathbf y)=c_1\int_0^\infty e^{-u}\mathbf H_{t^2\slash 4u}(\mathbf x,\mathbf y)\frac{du}{\sqrt{u}}, \ \ W_t(\mathbf x,\mathbf y) = c_1\int_0^\infty e^{-u}\mathbf S_{t^2\slash 4u}(\mathbf x,\mathbf y)\frac{du}{\sqrt{u}}.$$
Clearly, the maximal operator
\begin{equation*}\label{maximalW}
W_*f(\mathbf x)=\sup_{t>0} \Big|\int W_t(\mathbf x,\mathbf y)f(\mathbf y)
d\boldsymbol \mu(\mathbf y)\Big|
\end{equation*}
is bounded on $L^1(\mathbb R^n,\, d\boldsymbol\mu )$, that is,
\begin{equation}\label{App22}
\| W_*f\|_{L^1(\mathbb R^n,\, d\boldsymbol \mu )} \leq C \| f\|_{L^1(\mathbb R^n,\, d\boldsymbol \mu )}.
\end{equation}
Our task is to prove the following lemma.
\begin{lemma}\label{Lemma610}
 There are  constants $C_1,\, C_2,\, C_4,\,  \delta' >0$ such that
\begin{equation}\label{OnDiagonalLowerEstimateKr2}
U_t(\mathbf{x},\mathbf{x})\ge\tfrac{C_1}{\boldsymbol \mu (B (\mathbf x,t))};
\end{equation}
\begin{equation}\label{UpperEstimateKr2}
U_t(\mathbf{x},\mathbf{y})\le\tfrac{C_2}{ \boldsymbol \mu (B(\mathbf x,t))} \,\bigl(\ssf1\ssb
+\ssb\tfrac{\widetilde{d}\ssf(\mathbf{x},\ssf\mathbf{y})} { \boldsymbol \mu (B(\mathbf x,t))} \ssf\bigr)^{\ssb-1-\delta'};
\end{equation}

\begin{equation}\label{LipschitzEstimateKr2}
\bigl|\ssf U_t (\mathbf{x},\mathbf{y})\ssb
-\ssb U_t(\mathbf{x},\mathbf{y}^{\ssf\prime})\ssf\bigr|
\le\tfrac{C_4}{\boldsymbol \mu (B(\mathbf x, t))} \ssf\bigl(\ssf
1\ssb+\ssb\tfrac{\widetilde{d}\ssf(\mathbf{x},\ssf\mathbf{y})}{\boldsymbol \mu (B(\mathbf x,t))}
\ssf\bigr)^{\ssb-1-2\ssf\delta'}\ssf
\bigl(\ssf\tfrac{\widetilde{d}\ssf(\mathbf{y},\ssf\mathbf{y}^{\ssf\prime})}{\boldsymbol \mu ( B(\mathbf x,t))}
\ssf\bigr)^{\delta'}
\end{equation}
for $\widetilde{d}\ssf(\mathbf{y},\mathbf{y}^{\ssf\prime})\ssb
\le\ssb C_3\max\,\{\ssf \boldsymbol \mu (B(\mathbf x,t )),\widetilde{d}\ssf(\mathbf{x},\mathbf{y})\ssf\}$\ssf .
\end{lemma}
\begin{proof}
Take $0<\delta<\mathbf N^{-1}$.  By (\ref{OnDiagonalLowerEstimateKr}) and the subordination formula we have
 \begin{equation*}
 \begin{split}
 U_t(\mathbf x,\mathbf x)&\gtrsim \int_1^\infty  \frac{e^{-u}}{\boldsymbol \mu (B(\mathbf x,\frac{t}{2\sqrt{u}}))}
 \frac{du}{\sqrt{u}}
 \gtrsim  \int_1^\infty \frac{e^{-u}}{\boldsymbol \mu(B(\mathbf x, t))}  \frac{\boldsymbol \mu(B(\mathbf x, t))}{\boldsymbol \mu (B(\mathbf x,\frac{t}{2\sqrt{u}}))}
 \frac{du}{\sqrt{u}}\\
 &\gtrsim \frac{1}{\boldsymbol \mu(B(\mathbf x, t))} \int_1^\infty \frac{e^{-u}}{\sqrt{u}}\, du \gtrsim \frac{1}{\boldsymbol \mu(B(\mathbf x, t))},
 \end{split}
 \end{equation*}
 which proves (\ref{OnDiagonalLowerEstimateKr2}).

The proof of (\ref{UpperEstimateKr2}) is similar to that of (\ref{App7}). Indeed, by (\ref{UpperEstimateKr}) we have
\begin{equation}\begin{split}\nonumber
U_t(\mathbf x,\mathbf y)&\leq \int_0^\infty \frac{e^{-u}}{\boldsymbol \mu(B(\mathbf x, t))}
\frac{\boldsymbol \mu(B(\mathbf x, t))}{\boldsymbol \mu(B(\mathbf x, \frac{t}{2\sqrt{u}}))}
\Big( 1+ \frac{\tilde d(\mathbf x,\mathbf y)}{{\boldsymbol \mu(B(\mathbf x, t))}}\frac{\boldsymbol \mu(B(\mathbf x, t))}{\boldsymbol \mu(B(\mathbf x, \frac{t}{2\sqrt{u}}))}\Big)^{-1-\delta}\frac{du}{\sqrt{u}}\\
&\leq \int_0^{\frac{1}{4}} \frac{e^{-u}}{\boldsymbol \mu(B(\mathbf x, t))}
\frac{\boldsymbol \mu(B(\mathbf x, t))}{\boldsymbol \mu(B(\mathbf x, \frac{t}{2\sqrt{u}}))}
\Big( 1+ \frac{\tilde d(\mathbf x,\mathbf y)}{{\boldsymbol \mu(B(\mathbf x, t))}}\Big)^{-1-\delta} \Big( \frac{\boldsymbol \mu(B(\mathbf x, t))}{\boldsymbol \mu(B(\mathbf x, \frac{t}{2\sqrt{u}}))} \Big)^{-1-\delta}
\frac{du}{\sqrt{u}}\\
&\ \ +  \int_{\frac{1}{4}}^\infty  \frac{e^{-u}}{\boldsymbol \mu(B(\mathbf x, t))}
\frac{\boldsymbol \mu(B(\mathbf x, t))}{\boldsymbol \mu(B(\mathbf x, \frac{t}{2\sqrt{u}}))}
\Big( 1+ \frac{\tilde d(\mathbf x,\mathbf y)}{{\boldsymbol \mu(B(\mathbf x, t))}}\Big)^{-1-\delta}\frac{du}{\sqrt{u}}.\\
\end{split}\end{equation}
Now using (\ref{ComparisonVolumeBalls})   we obtain (\ref{UpperEstimateKr2}).

 Now we turn to the proof of (\ref{LipschitzEstimateKr2}). First we show that for every $\mathbf x,\mathbf y, \mathbf y'\in\mathbb R^n$ we see that
\begin{equation}\label{LipschitzEstimateKr3}
\bigl|\ssf U_t (\mathbf{x},\mathbf{y})\ssb
-\ssb U_t(\mathbf{x},\mathbf{y}^{\ssf\prime})\ssf\bigr|
\le\tfrac{C_4}{\boldsymbol \mu (B(\mathbf x, t))}
\bigl(\ssf\tfrac{\widetilde{d}\ssf(\mathbf{y},\ssf\mathbf{y}^{\ssf\prime})}{\boldsymbol \mu ( B(\mathbf x,t))}
\ssf\bigr)^{\ssb\frac1{\mathbf{N}}}.
\end{equation}
Since $U_t(\mathbf x,\mathbf y)\leq C\boldsymbol \mu (B(\mathbf x,t))^{-1}$ (see (\ref{UpperEstimateKr2})),  it suffices to prove (\ref{LipschitzEstimateKr3}) for $\tilde d(\mathbf y,\mathbf y')\leq \boldsymbol\mu (B(\mathbf x,t))$. Let $u_0\geq 1\slash 4$ be such that $\boldsymbol\mu(B(\mathbf x,\frac{t}{2\sqrt{u_0}}))=\tilde d(\mathbf y,\mathbf y')$. Then, using (\ref{LipschitzEstimateKr}) and (\ref{ComparisonVolumeBalls}), we have
\begin{equation}\begin{split}\nonumber
& \int_0^{u_0} e^{-u}|\mathbf H_{\frac{t^2} {4u}}(\mathbf x,\mathbf y)-\mathbf H_{\frac{t^2}{ 4u}}(\mathbf x,\mathbf y')|\frac{du}{\sqrt{u}}
 \lesssim  \int_0^{u_0} \frac{e^{-u}}{\boldsymbol \mu (B(\mathbf  x,\frac{t}{2\sqrt{u}}))} \Big(\frac{\tilde d(\mathbf y,\mathbf  y')}{\boldsymbol \mu(B(\mathbf x,\frac{t}{2\sqrt{u}}))}\Big)^{\frac{1}{ \mathbf N} } \frac{du}{\sqrt{u}}\\
 &= \int_0^{u_0} \frac{e^{-u}}{\boldsymbol \mu (B(\mathbf  x,t))} \Big(\frac{\tilde d(\mathbf y,\mathbf  y')}{\boldsymbol \mu(B(\mathbf x,t))}\Big)^{\frac{1}{ \mathbf N} }
 \Big(\frac{\mu (B(\mathbf  x,t))}{\mu (B(\mathbf  x,\frac{t}{2\sqrt{u}}))}\Big)^{1+\frac{1}{\mathbf N}}
 \frac{du}{\sqrt{u}}\\
 &\lesssim\frac{1}{\boldsymbol \mu (B(\mathbf  x,t))} \Big(\frac{\tilde d(\mathbf y,\mathbf  y')}{\boldsymbol \mu(B(\mathbf x,t))}\Big)^{\frac{1}{ \mathbf N} } \Big(\int_0^{1\slash 4} e^{-u}u^{n(1+\frac{1}{\mathbf N})\slash 2}\frac{du}{\sqrt{u}}+\int_{1\slash 4}^{u_0}e^{-u}u^{\mathbf N(1+\mathbf N^{-1})\slash 2}
 \frac{du}{\sqrt{u}}\Big)\\
 &\lesssim \frac{1}{\boldsymbol \mu (B(\mathbf  x,t))} \Big(\frac{\tilde d(\mathbf y,\mathbf  y')}{\boldsymbol \mu(B(\mathbf x,t))}\Big)^{\frac{1}{ \mathbf N} }.\\
\end{split}\end{equation}
Similarly, by (\ref{UpperEstimateKr}), we get
\begin{equation}\begin{split}\nonumber
 \int_{u_0}^\infty  e^{-u}|\mathbf H_{\frac{t^2} {4u}}(\mathbf x,\mathbf y)-\mathbf H_{\frac{t^2}{ 4u}}(\mathbf x,\mathbf y')|\frac{du}{\sqrt{u}}
 & \lesssim \int_{u_0}^\infty \frac{e^{-u}}{\boldsymbol \mu(B(\mathbf x,t))}\Big(\frac{\mu(B(\mathbf x,t))}{
 \mu(B(\mathbf x,\frac{t}{2\sqrt{u}}))}\Big)\frac{du}{\sqrt{u}} \\
 &\lesssim \frac{1}{\boldsymbol \mu(B(\mathbf x,t))} \int_{u_0}^\infty e^{-u}u^{\mathbf N\slash 2}\frac{du}{\sqrt{u}}\\
 &\lesssim \frac{1}{\boldsymbol \mu(B(\mathbf x,t))}u_0^{-\mathbf N\slash 2}.
 \end{split}\end{equation}
Since
$$ \frac{\tilde d(\mathbf y,\mathbf y')}{ \boldsymbol\mu ( B(\mathbf x,t))}=
\frac{\boldsymbol\mu (B(\mathbf x,\frac{t}{2\sqrt{u_0}}))}{\boldsymbol \mu (B(\mathbf x,t))}\gtrsim
u_0^{-\mathbf N\slash 2},$$
(see (\ref{ComparisonVolumeBalls})), we obtain (\ref{LipschitzEstimateKr3}).

We are now in a position to continue the proof of (\ref{LipschitzEstimateKr2}).

If $\tilde d(\mathbf x,\mathbf y)\leq \boldsymbol\mu (B(\mathbf x,t))$ then (\ref{LipschitzEstimateKr2}) follows from (\ref{LipschitzEstimateKr3}).

If $\tilde d(\mathbf x,\mathbf y)>\boldsymbol\mu (B(\mathbf x,t))$
and $\tilde d(\mathbf y,\mathbf y') <\tilde d(\mathbf x,\mathbf y)\slash (2A)$, then $\tilde d(\mathbf x,\mathbf y)\leq 2A\tilde d (\mathbf x,\mathbf y')$. Hence, from (\ref{UpperEstimateKr2})
we conclude that
\begin{equation}\label{LipschitzEstimateKr4}
\bigl|\ssf U_t (\mathbf{x},\mathbf{y})\ssb
-\ssb U_t(\mathbf{x},\mathbf{y}^{\ssf\prime})\ssf\bigr| \leq
\tfrac{C_2'}{\boldsymbol \mu (B(\mathbf x, t))} \ssf\bigl(\ssf
1\ssb+\ssb\tfrac{\widetilde{d}\ssf(\mathbf{x},\ssf\mathbf{y})}{\boldsymbol \mu (B(\mathbf x,t))}
\ssf\bigr)^{\ssb-1-\ssf\delta}\ssf .
\end{equation}
 Consequently, we deduce (\ref{LipschitzEstimateKr2}) (with perhaps small $\delta'>0$) from (\ref{LipschitzEstimateKr3}) and (\ref{LipschitzEstimateKr4}).

 It remains to consider the case when   $\tilde d(\mathbf x,\mathbf y) > \boldsymbol\mu (B(\mathbf x,t))$
and $\tilde d(\mathbf y,\mathbf y') \geq \tilde d(\mathbf x,\mathbf y)\slash (2A)$.  Recall that  $\tilde d(\mathbf y,\mathbf y')\leq \boldsymbol \mu (B(\mathbf x,t))$. Thus $\tilde d(\mathbf x,\mathbf y)\sim \boldsymbol\mu (B(\mathbf x,t))$. So, finally, using (\ref{LipschitzEstimateKr3}) we have
\begin{equation}\nonumber
\bigl|\ssf U_t (\mathbf{x},\mathbf{y})\ssb
-\ssb U_t(\mathbf{x},\mathbf{y}^{\ssf\prime})\ssf\bigr|
\leq
\tfrac{C_4}{\boldsymbol \mu (B(\mathbf x, t))}
\bigl(\ssf\tfrac{\widetilde{d}\ssf(\mathbf{y},\ssf\mathbf{y}^{\ssf\prime})}{\boldsymbol \mu ( B(\mathbf x,t))}
\ssf\bigr)^{\ssb\frac1{\mathbf{N}}}
\leq
\tfrac{C_4}{\boldsymbol \mu (B(\mathbf x, t))} \bigl(\ssf\tfrac{\widetilde{d}\ssf(\mathbf{y},\ssf\mathbf{y}^{\ssf\prime})}{\boldsymbol \mu ( B(\mathbf x,t))}
\ssf\bigr)^{\ssb\frac1{\mathbf{N}}}\ssf\bigl(\ssf
1\ssb+\ssb\tfrac{\widetilde{d}\ssf(\mathbf{x},\ssf\mathbf{y})}{\boldsymbol \mu (B(\mathbf x,t))}
\ssf\bigr)^{\ssb-1-\ssf\delta}\ssf.
\end{equation}
This completes the proof of Lemma \ref{Lemma610}.
\end{proof}
Set $K_r(\mathbf x,\mathbf y)=U_t(\mathbf x,\mathbf y)$, where $r=\boldsymbol\mu (B(\mathbf x,t))$. Now part (a) of Theorem \ref{propositionP} follows from (\ref{splitP}), boundedness of the maximal function  $W_*$ on $L^1(\mathbb R^n, \boldsymbol \mu)$, and the Uchiyama theorem (see Theorem \ref{TheoremUchiyama}) combined with  Lemma \ref{Lemma610}.

 Now we turn to the proof of part (b) of Theorem \ref{propositionP}. Recall that $P_t(\mathbf x,\mathbf y)>0$.   So, by  (\ref{App5}), the operator $P_*$ is bounded on $L^\infty(\mathbb R^n, d\boldsymbol \mu)$.   Thanks to (\ref{UpperEstimateKr2}) and (\ref{App22}), it is of weak-type (1,1). Finally, from the Marcinkiewicz interpolation theorem we conclude that  $P_*$ is bounded on $L^p(\mathbb R^n,\, d\boldsymbol \mu)$ for $1<p<\infty$.
\end{proof}

\begin{proof}[Proof of Proposition \ref{propositionP2}] Fix $\varepsilon >0$. There is $R>0$ such that $|g(\mathbf x)|<\varepsilon $ for $|\mathbf x|>R$. Write
$$ g=g\chi_{B(0,R)}+g\chi_{B(0,R)^c}=: g_0+ g_1.$$
From (\ref{App5}) we get $|P_tg_1(\mathbf x)|<\varepsilon$ for every $t>0$ and $\mathbf x\in\mathbb R^n$.
Now using (\ref{App6}) we obtain
$$ | P_tg_0(\mathbf x)| \leq \frac{C}{\boldsymbol \mu (B(\mathbf x, t))} \| g_0\|_{L^1(\mathbb R^n, \, d\boldsymbol \mu)}\to 0 \ \ \ \text{as}\ t\to \infty.$$
On the other hand, if $t$ remains in a bounded interval and $|\mathbf x|>2nR$, applying (\ref{App7})  we have
$$ |P_tg_0(\mathbf x)| \leq \frac{C}{\boldsymbol \mu(B(\mathbf x, t))} \Big(1+\frac{\boldsymbol\mu (B(\mathbf x, |\mathbf x|))}{\boldsymbol \mu(B(\mathbf x, t))}\Big)^{-1-\delta} \| g_0\|_{L^1(\mathbb R^n, \, d\boldsymbol \mu)}\to 0 \ \ \text{as}\ |\mathbf x|\to\infty.$$
The proof of the first part of  Proposition \ref{propositionP2} is complete.

In order to prove the second part of the proposition we fix $\varepsilon >0$. We claim that
$$\lim_{|\mathbf x|\to\infty} P_\varepsilon f(\mathbf x) =0.$$
To proof the claim let $\varepsilon' >0$. Take $R>0$ large enough such that $\int_{|\mathbf y|>R} |f(\mathbf y)|\, d\boldsymbol \mu (\mathbf y)\leq \varepsilon' \boldsymbol\mu(B(0, \varepsilon))$. Write $f=f\chi_{B(0,R)}+f\chi_{B(0,R)^c}=:f_0+f_1$.
Then, by (\ref{App5}) and (\ref{App6}) we have  $|P_\varepsilon f_1|\leq \varepsilon'$. On the other hand  from the first part of the proposition we conclude that $\lim_{|\mathbf x|\to\infty } P_\varepsilon f_0(\mathbf x)=0$, which gives the claim. Now (\ref{DD1}) follows from the first part of Proposition \ref{propositionP2}, since
$P_{t+\varepsilon} f=P_t(P_\varepsilon f)$.
\end{proof}
{\bf Acknowledgments.} The author wishes to thank Jean-Philippe Anker, Pawe{\l}  G{\l}owacki and Rysiek Szwarc for their remarks. The author is greatly indebt  Bartosz Trojan for his suggestions which shorter the original proof of Theorem \ref{TheoremBessel}.


\begin{thebibliography}{99}



\bibitem{ABDH} J.-Ph. Anker, N. Ben Salem, J. Dziuba\'nski, N. Hamda,
{\it The Hardy space $H^1$ in the rational Dunkl setting}, to appear in Constr. Approx. arXiv:1309.5567.

\bibitem{BurkholderGundySilverstein}
D.L.~Burkholder, R.F.~Gundy, M.L.~Silverstein,
\textit{A maximal function characterisation of the class $H^p$},
Trans. Amer. Math. Soc. \textbf{157} (1971), 137--153


\bibitem{Coifman}
R.R.~Coifman,
\textit{A real variable characterization of $H^p$},
Studia Math. \textbf{51} (1974), 269--274

\bibitem{CoifmanWeiss}
R.R.~Coifman, G.L.~Weiss,
\textit{Extensions of Hardy spaces and their use in analysis\/},
Bull. Amer. Math. Soc. \textbf{83} (1977), 569--615

\bibitem{Dunkl0}
C.F.~Dunkl,
\textit{Reflection groups and orthogonal polynomials on the sphere\/},
 Math. Z.  \textbf{197} (1988),  33--60


\bibitem{Dunkl}
C.F.~Dunkl,
\textit{Differential-difference operators associated to reflection groups\/},
Trans. Amer. Math. \textbf{311} (1989), 167--183

\bibitem{Dunkl2}
C.F.~Dunkl,
\textit{Integral kernels with reflection group invariance \/},
Canad. J. Math. \textbf{43} (1991), 1213--1227

\bibitem{Dunkl3}
C.F.~Dunkl,
\textit{Hankel transforms associated to finite reflection groups\/},
in: {\it Proc. of the special session on hypergeometric functions on domains of positivity, Jack polynomials and applications.} Proceedings, Tampa 1991, Contemp. Math.  \textbf{138} (1989), pp. 123-138

\bibitem{DziubanskiPreisnerWrobel}
Dziuba\'nski, J., Preisner, M., Wr\'obel, B.\ssf:
Multivariate H\"ormander-type multiplier theorem for the Hankel transform.
J. Fourier Anal. Appl. \textbf{19} (2), 417--437 (2013)



\bibitem{Evans} L.C. Evans, \textit{Partial Differential Equations}, Graduate Studies in Mathematics vol. 19, AMS (1998).

\bibitem{FeffermanStein}
C.~Fefferman, E.M.~Stein,
\textit{$H^p$ \ssb spaces of several variables\/},
Acta Math. \textbf{129} (1972), 137--193


\bibitem{dJ} M.F.E. de Jeu, \textit{The Dunkl transform}, Invent. Math. \textbf{113} (1993), 147--162.

\bibitem{MaciasSegovia}
R.A.~Mac\'ias, C.~Segovia,
\textit{A decomposition into atoms of distributions on spaces of homogeneous type\/},
Adv. in Math. \textbf{33} (1979), 271--309

\bibitem{MucStein} B. Muckenhoupt, E. Stein, {\it Classical expansions and their relation to conjugate harmonic
functions}, Trans. Amer. Math. Soc. 118 (1965), 17-–92

\bibitem{NIST}
NIST Digital Library of Mathematical Functions.
http\ssf:\ssf/\!/dlmf.nist.gov


\bibitem{Opdam}
E.M.~Opdam,
\textit{Lecture notes on Dunkl operators
for real and complex reflection groups\/},
Math. Soc. Japan Mem. \textbf{8} (2000)



\bibitem{Roesler2}
M.~R\"osler,
\textit{Generalized Hermite polynomials and the heat equation for Dunkl operators\/},
Comm. Math. Phys. \textbf{192} (1998), 519--542

\bibitem{Roesler3}
M.~R\"osler,
\textit{Dunkl operators}: \textit{theory and applications\/},
in \textit{Orthogonal polynomials and special functions\/} (\textit{Leuven, 2002\/}),
Lect. Notes Math. \textbf{1817}, Springer-Verlag (2003), 93--135



\bibitem{RoeslerVoit}
M.~R\"osler, M. Voit,
\textit{ Dunkl theory, convolution algebras, and related Markov processes\/},
in \textit{Harmonic and stochastic analysis of Dunkl processes}, Collection Travaux en cours, 71, Hermann, Paris 2008, pp. 1--112.



\bibitem{Stein}
E.M.~Stein,
\textit{Harmonic Analysis}
(\textit{Real\ssf-\ssb Variable Methods, Orthogonality, and Oscillatory Integrals\/}),
Princeton Math. Ser. \textbf{43}, Princeton Univ. Press (1993)


\bibitem{SteinSing}
E.M.~Stein,
\textit{Singular integrals and differetiability of functions}, Princeton Univ. Press, Princeton (1971)



\bibitem{SteinWeiss}
E.M.~Stein, G.L.~Weiss,
\textit{On the theory of harmonic functions of several variables I}
(\textit{the theory of $H^p$ spaces\/}),
Acta Math. \textbf{103} (1960), 25--62

%\bibitem{Soni}
%R.P.~Soni,
%\textit{On an inequality for modified Bessel functions\/},
%J. Math. Phys. \textbf{44} (1965), 406--407



\bibitem{Uchiyama}
A.~Uchiyama,
\textit{A maximal function characterization of $H^p$
on the space of homogeneous type\/},
Trans. Amer. Math. Soc. \textbf{262} (1980), no. 2, 579--592

\bibitem{Watson}
G.N.~Watson,
\textit{A treatise on the theory of Bessel functions\/},
Cambridge Univ. Press (1995).

\end{thebibliography}
\end{document}